\documentclass[10pt]{amsart}

\usepackage{amssymb,amsfonts,amsthm,amscd,stmaryrd}

\usepackage[initials]{amsrefs}
\usepackage{graphicx}

\usepackage{color}
\usepackage{ae}
\usepackage[utf8]{inputenc}
\numberwithin{figure}{section}

\usepackage{mathrsfs,a4wide}

\usepackage[final,allcolors=blue,colorlinks=true]{hyperref}
\usepackage[leqno]{amsmath}

\newtheorem{theorem}{Theorem}[section]
\newtheorem{lemma}[theorem]{Lemma}
\newtheorem{proposition}[theorem]{Proposition}

\theoremstyle{definition}
\newtheorem{definition}[theorem]{Definition}

\newtheorem{remark}[theorem]{Remark}
\numberwithin{equation}{section}

\newcommand{\abs}[1]{\lvert #1 \rvert}
\newcommand{\R}{\mathbb{R}}
\newcommand{\C}{\mathbb{C}}
\newcommand{\N}{\mathbb{N}}
\newcommand{\Si}{\mathbb{S}^1}

\newcommand{\Ma}{\mathbb{M}}
\newcommand{\To}{\mathbb{T}}

\newcommand{\dist}{{\rm dist}}

\newcommand{\eps}{\varepsilon}

\renewcommand{\sharp}{{\#}}

\newcommand{\norm}[1]{{\lVert #1\rVert}}

\newcommand{\F}{{\mathcal{F}}}
\newcommand{\Ha}{{\mathcal{H}}}
\newcommand{\A}{{\mathcal{A}}}

\newcommand{\diver}{\operatorname{div}}
\newcommand{\vect}[1]{\underline{#1}}

\newcommand{\lip}{\mathrm{Lip}}
\newcommand{\loc}{\mathrm{loc}}
\newcommand{\reg}{\mathrm{reg}}
\newcommand{\ds}{\displaystyle}

\newcommand{\id}{\mathrm{id}}

\renewcommand{\MR}[1]{\null}

\begin{document}

\title[A minimality criterion for voids in elastic bodies]
{A quantitative second order minimality criterion \\for cavities in elastic bodies}

\author{Giuseppe Maria Capriani}
\address{Dipartimento di Matematica e Applicazioni ''R. Cacciopoli'', Università degli Studi di Napoli ''Federico II'', Napoli, Italy}
\email{giuseppe.capriani@gmail.com}

\author{Vesa Julin}
\address{Dipartimento di Matematica e Applicazioni ''R. Cacciopoli'', Università degli Studi di Napoli ''Federico II'', Napoli, Italy}
\email{vesa.julin@jyu.fi}

\author{Giovanni Pisante}
\address{Dipartimento di Matematica, Seconda Università di Napoli, Caserta, Italy}	
\email{giovanni.pisante@unina2.it}

\date{\today}

\begin{abstract}
We consider a functional which models an elastic body with a cavity. We show that if
a critical point has positive second variation then it is a strict local minimizer. We also provide a quantitative
estimate.
\end{abstract}
\maketitle

\section{Introduction}

The role of roughness appearing onto the surfaces and interfaces of nano-structures has been proved to be of great
significance in several fields such as micro-electronics, metallurgy and materials science. For instance the roughness
can strongly modify the mechanical properties of multilayered structures as confirmed by the observation that
dislocations, islands and cracks can be generated from a rough surface (see \cite{Colin}). Many efforts have been
devoted to the investigation on how to control the roughness appearing onto the surfaces and interfaces of
nano-structures, leading to the study of the so-called Driven Rearrangement Instability, i.e., the  morphological
surfaces
instability of interfaces between solids generated by elastic stress. This phenomenon has been detected, for instance,
in
hetero-epitaxial growth of thin films with a lattice mismatch between film and substrate and in stressed elastic solids
with cavities.

The theoretical investigation of the stability of the free surface of a planar non-hydrostatically stressed solid has
been performed in the pioneering papers by Asaro and Tiller \cite{Asar} and Grinfeld \cite{grinfeld-2}.
These authors showed that the free surface is unstable with respect to a given family of sinusoidal fluctuations. They
also
gave a first insightful description of the phenomenon, nowadays named Asaro-Grinfeld-Tiller instability, in
which a thin film growing on a flat substrate remains flat up to a critical value of the thickness, after which, the free surface becomes unstable developing corrugations and irregularities.
This instability is explained as a consequence of the presence of two competing energies, usually identified with a bulk
elastic energy and a surface energy. After these results the interest of the scientific community on the rigorous
mathematical study of the morphological instabilities has rapidly grown. Starting from the paper \cite{grinfeld}
where Grinfeld  follows the Gibbs variational approach to model the morphology of thin films, it became clear that a
second order variational analysis could be successfully used.
This approach has been used in the context of epitaxial growth first for a one dimensional model in \cite{Bonne}.
Then in \cite{Bonne-2} and \cite{Fonseca:2007bj} the
model introduced in \cite{grinfeld}, which is a more realistic two-dimensional model, corresponding to three-dimensional
configurations with planar symmetry, is studied and the problem of finding a proper functional setting is successfully addressed. This settled the framework in which a precise and detailed analysis of
qualitative properties of regular equilibrium configurations has been carried out by Fusco and Morini in \cite{Fusco:2009ug} via
a second order variational analysis. Indeed they prove a sufficient condition for local minimality in terms of the
positivity of second variation and provide a sufficiently complete picture of the phenomena that occur in
epitaxially-growing thin films.

Such detailed analysis is instead far from being complete in the framework of stressed elastic solids with cavities. In this paper we perform a second order variational analysis for a two-dimensional variational model that has been recently used to
describe surface instability in morphological evolution of cavities in stressed solids (see for instance
\cites{Gao,sieg-miks-voor,Wang}) with the aim of deriving new minimality conditions for equilibria and studying their stability.
The model can be roughly described as follows. Consider a cavity in an elastic solid, that
will be identified with a smooth compact set $F\subset \R^{2}$, starshaped with respect to the origin. The solid region
is assumed to obey to the classical law of linear elasticity, so that the bulk energy can be written in the form 
\[
 \int_{B_{R_0} \backslash F} Q(E(u)) \, dz,
\]
where  $E(u)$ is the symmetric gradient of the elastic displacement $u$ and $Q$ is a bilinear form depending on the
material (see Section \ref{Preliminaries} for details). The surface energy is simply assumed to be the length of the
boundary of $F$. Then the energy for a regular configuration is expressed by the functional
\begin{equation*}
 \F(F,u) :=  \int_{B_0 \backslash F} Q(E(u)) \, dz +  \Ha^{1}(\partial F )\,.
\end{equation*}
In this framework the shape of the void plays a key role in the evolution of cavities in stressed solid bodies, while
the effects of the volume changes are negligible. Hence, one usually assumes that the void evolves preserving its
volume.
The equilibria are therefore identified with minimizers of $\F(F,u)$ under the volume constraint $\abs{F}=d$. Since admissible configurations need not to be regular, the energy of such configurations has to be defined via a relaxation procedure. This issue, together with the study of the regularity of minima, has been addressed (even for
more general functionals involving anisotropic surface energies) in \cite{FFLmill} where, in order to keep track of the
possible appearance of cracks, the relaxed functional with respect to the Hausdorff convergence has been studied. The
relaxed functional can be expressed in the following form:
\begin{equation}\label{functional-relaxed}
 \F(F,u) :=  \int_{B_0 \backslash F} Q(E(u)) \, dz +  \Ha^{1}(\Gamma_{F})+2 \Ha^{1}(\Sigma_{F})\,,
\end{equation}
where $F$ has finite perimeter, $\Gamma_{F}$ is the ``regular'' part of $\partial F$ and $\Sigma_{F}$ represents the
cracks (see Section
\ref{Preliminaries}).  

The main result of the paper is a quantitative minimality criterion that relies on the study of the second variation of the functional \eqref{functional-relaxed}.  To be more precise we prove in Theorem \ref{TH:criterion} that if $(F,u)$ is a smooth critical configuration and the non local quadratic form 
$\partial^{2}\F(F,u)$ associated to the second variation of $\F$ at $(F,u)$    is  positively defined, then there exists
a constant $c_{0}$ such that 
\begin{equation}\label{result}
\F(G,v) > \F(F,u) +  c_0 \abs{G\Delta F} ^2
\end{equation}  
for any given admissible configuration $(G,v)$ 
with $G$ sufficiently close to $F$ in the Hausdorff distance
and $G\neq F$. In particular this implies not only that $(F,u)$ is a strict local minimizer of
\eqref{functional-relaxed} but also
provides a quantitative estimate of the deviation from minimality for configurations close to $(F,u)$ in the spirit
of the recent result obtained in \cite{AcFuMo}. The minimality criterion is then applied to the case of a disk subjected to radial stretching where the second variation can be explicitly estimated to prove the local
and global minimality of the round
configuration if the applied stress is sufficiently small.

We point out that an important open problem is how to remove the assumption of starshapedness. Indeed, even the explicit
form of the relaxed functional is unknown.

We conclude by outlining the structure of the paper and making some comments about the proofs. In Section \ref{sec:2nd} we
calculate the second variation of $\F$ at any regular configuration (see Theorem \ref{2d-var:formula}) and we exploit
the volume constraint to define the associated quadratic form in a critical configuration. At the end of the section in
Lemma \ref{easylemma} we prove a ``weak'' coercivity property of $\partial^{2}\F(F,u)$ in a critical point, which is the
first step towards the proof of Theorem \ref{TH:criterion}. In
Section \ref{Sec:4}, as an intermediate step, we prove that the positivity of the second variation implies the local
minimality among configurations $(G,v)$ for $G$  close to $F$ in the $C^{1,1}$-topology. The main point in achieving
this result is to overcome the lack of
$C^{1,1}$-coercivity, which would immediately imply the result. This is done by proving the
stability of the weak coercivity with respect to a one-parameter perturbation of the critical configuration (see
Lemma \ref{2ndpostive}). In section \ref{Sec:5} we exploit the regularity theory for a class of obstacle problems which
arise as perturbations of \eqref{functional-relaxed}  to show that the $C^{1,1}$-minimality actually implies the
minimality with respect to the Hausdorff distance, thus proving the theorem. In the last section we apply the previous
analysis to the explicit case of a disk subjected to a radial stretching.

\section{Preliminaries}\label{Preliminaries}
In this section we fix the notation and describe precisely the required background for our analysis.
We are interested in cavities identified as closed sets $F$ with $\Ha^1(\partial F)<+\infty$ and starshaped with respect
to the origin. The fact that $F$ is starshaped allows us to describe it as a subgraph of a function. Since $F$
has finite perimeter, the function associated to its boundary turns out to have bounded pointwise total variation. This
will allow us to deal with functions rather than sets.

We denote by $\Si$ the unit circle in $\R^2$ and by $\sigma:\R\to\Si$ the local diffeomorphism defined by
$\sigma(\theta)=(\cos \theta,\sin \theta)$, by $\sigma^{-1}$ its local inverse and by
$\sigma^\perp(\theta)=(\sin\theta,-\cos\theta)$ its orthogonal. We set $C^2_\sharp(\R)$ to be the
collection of functions in $C^2(\R)$ that are $2\pi$-periodic. In a similar way we shall define the function spaces
$H^1_\sharp(\R)$, etc.

With a slight abuse of notation we set 
\begin{equation}
\label{BV_sharp}
BV_{\sharp}(\R) := \{  g: \R \to [0,R_0] \mid g \,\, \text{is upper semicontinuous}, \, 2\pi\text{-periodic and}\,\,
pV(g, [0, 2 \pi])
<
\infty\},
\end{equation}
where $pV(g, [0, 2 \pi])$  is the pointwise total variation of $g$ in $[0, 2 \pi]$ and $R_0$ is the radius of a large
ball $B_{R_{0}}$. For
a function $g \in
BV_{\sharp}(\R)$ we define the
extended graph
of $g$ as $\Gamma_g\cup \Sigma_g$, where
\begin{equation}
\label{Gamma_g}
\Gamma_g := \{ \rho \sigma(\theta)\in\R^2 \mid  g^-(\theta) \leq \rho \leq g^+(\theta),\, \theta \in \R \}
\end{equation}
and 
\begin{equation}
\label{Sigma_g}
\Sigma_g = \overline{\{  \rho \sigma(\theta)\in\R^2 \mid   g^+(\theta) < \rho < g(\theta),\, \theta \in \R \}}\,.
\end{equation}
Here $g^-(\theta) :=  \liminf_{\tilde{\theta} \to \theta}g(\tilde{\theta})$ and $g^+(\theta) :=  \limsup_{\tilde{\theta}
\to \theta}g(\tilde{\theta})$. We shall refer to $\Sigma_g$ as the \textit{set of cracks}.

Let us consider  a compact set $F\subset\overline{B}_{R_0}$ starshaped with respect to the origin. Then, for
$\sigma\in\Si$, we
can write
\[
F=\{r\sigma(\theta)\in\R^2 \mid \theta\in\R,\, 0\leq r\leq \rho_F(\theta)\}\,, 
\]
where $\rho_F$ is the \textit{radial function} of $F$ and is defined by
\[
\rho_F(\theta) := \sup \,  \{ \rho\in\R \mid \rho \sigma (\theta) \in F\}\,.
\]
It is clear that $\rho_F: \R \to [0, R_0]$ is upper semicontinuous. Moreover we have the following result, see 
\cite{FFLmill}*{Lemmata 2.2 and 2.3}.  
\begin{lemma}
Let $F \subset \bar{B}_{R_0}$ be a closed set starshaped with respect to the origin and let $\rho_F$ be the radial
function of $F$. Then 
\[
\partial F = \Gamma_{\rho_F} \cup \Sigma_{\rho_F}.
\]
Moreover $\Ha^1(\partial F) < + \infty$ if and only if $\rho_F$ has finite pointwise total variation. 
\end{lemma}

The previous lemma rigorously shows that we may use radial functions instead of sets. Hence, for $g \in BV_{\sharp}(\R)$
we set 
\[
F_g := \{ \rho \sigma(\theta)\in\R^2  \mid 0 \leq \rho \leq g(\theta)\}\quad \text{\ and\ } \quad
\Omega_g := B_{R_0} \setminus F_g.
\]
We may think of $F_g$ as the void and of $\Omega_g$ as the elastic solid.

We can now define properly the space of admissible pairs. Given $u_0\in C^\infty(\R^2\setminus B_{R_0})$ we set
\begin{equation}
\label{admissible.pair}
X(u_0) = \{   (g,v) \mid g  \in BV_{\sharp}(\R), \, v \in H_{\loc}^1(\R^2 \setminus F_g; \R^2), \, v \equiv u_0 \,\,
\text{outside} \,\, B_{R_0} \}\,,
\end{equation}
and we shall use the notation $X(0)$ for $u_0 \equiv 0$. We define also the following subspaces of $X(u_0)$ 
\begin{equation}
\label{admissible.regular}
\begin{split}
X_{\lip}(u_0) &:= \{   (g,v)  \in X(u_0) \mid g\,\, \text{is Lipschitz} \}, \\
X_{\reg}(u_0) &:= \{   (g,v) \in X(u_0) \mid g  \in C^{\infty}_\sharp(\R), \, v \in C^{\infty}(\bar{\Omega}_g) \}.
\end{split}
\end{equation}
We are now in position to give the proper definition  of convergence in $X(u_0)$.
\begin{definition}
\label{conergence.X}
A sequence $(g_n,v_n) \subset X(u_0)$ is said to converge to $(g,v)$ in $X(u_0)$ and we write $(g_n,v_n) 
\overset{X}{\longrightarrow}(g,v)$ if
\begin{enumerate}
\item $\displaystyle \sup_{n\in\N}\Ha^{1}(\partial F_{g_n})<+\infty\,,$
\item $\displaystyle  F_{g_n}\to F_g$ in Hausdorff metric,
\item $\displaystyle v_{n}\rightharpoonup v$ weakly in $H^{1}(\omega;\R^{2})$ for any open set $\omega$
compactly contained in $ \R^{2}\setminus F_g$.
\end{enumerate} 
\end{definition}

In view of \cite{FFLmill}*{Lemma 2.6}, we see that $X(u_0)$ is closed under the convergence of Definition
\ref{conergence.X}.  

The \textit{elastic energy density} is defined by $Q(E(u)):= \frac{1}{2}\C E(u):E(u)$, where $\C$ is the fourth order
tensor
\[
\C \xi := \begin{pmatrix} (2 \mu + \lambda )\xi_{11} + \lambda \xi_{22} & 2 \mu \xi_{12} \\
			   2 \mu \xi_{12}  & (2 \mu + \lambda )\xi_{22} + \lambda \xi_{11}
\end{pmatrix}
\]
and $E(u)$ is the symmetric gradient of $u$
\[
E(u) := \frac{1}{2}(Du + (Du)^T).
\]
The constants $\mu, \lambda$ are called the \textit{Lamé coefficients} and they are assumed to satisfy the
following ellipticity conditions 
\[
\mu >0 \quad \text{and} \quad \lambda > - \mu.
\]
Since $Q(\xi) \geq \min\{\mu, \mu + \lambda \} |\xi|^2$ for every symmetric $2 \times 2$ matrix $\xi$, the above
conditions guarantee that $Q$ is coercive. We also set the ellipticity constant
\[
\eta :=  \min\{\mu, \mu + \lambda \}.
\]

For a pair $(g,v) \in X_\lip(u_0)$ we may write the value of the functional \eqref{functional-relaxed} as
\[
\F(g,v) = \int_{\Omega_g} Q(E(v))\, dz +\Ha^1(\Gamma_g).
\]
Since this functional is not lower semicontinuous with respect to the convergence in $X(u_0)$, in order to
effectively address the minimization problem we consider the relaxed functional
\[
\bar{\F}(g,v)= \inf \{  \liminf_{n \to \infty} \F(g_n,v_n) \mid  (g_n,v_n) \in X_{\lip}(u_0), \,  (g_n,v_n) 
\overset{X}{\longrightarrow}(g,v) \}.
\]
The following integral representation of $\bar{\F}$ is proved in \cite{FFLmill}*{Theorem 3.1}, where the more general
case of anisotropic surface energy is also considered.
\begin{theorem}
\label{relaxed.f}
Let $(g,v) \in X(u_0)$, then 
\[
\bar{\F}(g,v)= \int_{\Omega_g} Q(E(v))\, dz +\Ha^1(\Gamma_g) + 2 \Ha^1(\Sigma_g).
\]
\end{theorem}

From now on we will always deal with the relaxed functional appearing in Theorem \ref{relaxed.f} and with abuse of
notation we will denote it simply by $\F(g,v)$. The minimization problem can now be properly stated as
\begin{equation}
\label{min.problem}
\min \{ \F(g,v) \mid (g,v) \in X(u_0) , \,\, |\Omega_g| = d\}
\end{equation}
for some given constant $d < |B_{R_0}|$. Existence of solutions of the problem \eqref{min.problem} is then ensured by
\cite{FFLmill}*{Theorem 3.2}.

Given $g \in BV_{\sharp}(\R)$ there is one
particular elastic displacement $v$ which is the minimizer of the elastic energy $\int_{\Omega_g} Q(E(v))\, dz$ under
the
boundary condition $v \equiv u_0$ outside $B_{R_0}$. We call this map \emph{the elastic equilibrium} associated to $g$.
If $(h,u) \in X(u_0)$ solves \eqref{min.problem} then $u$ has to be the elastic equilibrium
associated with $h$. 

Assume now that a solution $(h,u)$ belongs to $X_\reg(u_0)$ and $h>0$, then
$(h,u)$ satisfy the Euler-Lagrange equations 
\begin{equation}
\label{euler}
\left\{
\begin{aligned}
\diver \C(E(u)) &= 0 &&\text{in }\Omega_h \\
\C(E(u))[\nu] &= 0  &&\text{on } \Gamma_h \\
 Q(E(u)) - k_h &= \text{const.} &&\text{on\ } \Gamma_h,\\
\end{aligned}
\right.
\end{equation}
where $k_h$ is the curvature of $\Gamma_h$. The first two equations are standard whereas the third one is the first
variation of the functional \eqref{functional-relaxed}. This motivates the following definition.
\begin{definition}
\label{critical}
A pair $(h,u)\in X_{\reg}(u_0)$ is said to be critical if  it solves the equations \eqref{euler}.
\end{definition}

We remark that if $(h,u)$ is a critical pair, then from the first two equations in \eqref{euler} it follows
that $u$ is the elastic equilibrium associated to $h$. We also point out that in the definition of a critical point we
only need to assume $h$ to be smooth. Indeed, if we only assume $(h,u)\in X(u_0)$ and $h \in C^{\infty}(\R)$, then it
follows from the standard elliptic regularity theory (see \cite{agmon}) that $u \in C^{\infty}(\bar{\Omega}_h)$.

The regularity for minimizers of  \eqref{min.problem} was studied in \cite{Fonseca:2007bj} and
the following result holds. If the pair $(h,u)$ is a local minimizer of \eqref{min.problem} and $0<h<R_0$ then there
exists an open set $I \subset [0, 2 \pi)$ of full measure such that  $h \in C^{\infty}(I)$. In fact $h$ is even analytic
in $I$. Hence our regularity  assumption on a critical point in Definition \ref{critical}  is not restrictive when
$0<h<R_0$ and the singular set is empty.   

Finally, we recall a version of the Korn's inequality which will be used throughout the paper, see
e.g.\ \cite{korn-horgan}.
\begin{theorem}[Korn's inequality]
\label{kornann}
Let $\Omega \subset \R^2$ be a bounded domain with smooth boundary and $v  \in W^{1,2}(\Omega; \R^2)$. There exists a
constant $C = C(\Omega)$ such that if 
\[
\int_{\Omega} Dv \, dz =  \int_{\Omega} Dv^T \, dz,
\]
then 
\[
\int_{\Omega} |Dv|^2 \, dz \leq C \int_{\Omega} |E(v)|^2 \, dz.
\]

Moreover if $\Omega$ is an annulus $A(R,r)$, $R>r$, the constant $C$ depends only on the ratio $r/R$ and $C \to
4$
as $r/R \to 0$.
\end{theorem}

\section{Calculation of the second variation}\label{sec:2nd}

The goal of this section is to calculate the second variation of the functional $\F$ 
at any point $(h,u) \in X_{\reg}(u_0)$, where $u$ is the elastic equilibrium associated to $h$ and 
$0<  h < R_0$, see formula \eqref{2ndFormula}. We then define a quadratic form
for a critical pair \eqref{biliear} and give a definition of positiveness of the second variation, 
see Definition \ref{pos.variat}.

To this aim we will introduce the following notation. Given a $2\pi$-periodic function $f:\R\to \R$ we will
denote by $\vect{f}:\R^{2}\setminus\{0\}\to \R^{2}$ the map
\begin{equation}
\label{not.vect}
\vect{f}(z):= f\left( \sigma^{-1}\left( \frac{z}{|z|}\right) \right)  \frac{z}{|z|}.
\end{equation}

For a parameter $s\in(-\eps,\eps)$ let $(h_s, u_s)  \in X_{\reg}(u_0)$ be a smooth perturbation of $(h,u)$, where $u_s$
is the elastic equilibrium
associated to $h_s$. By smooth perturbation we mean that the function $(s,\theta) \mapsto h_s(\theta)$ is
smooth and $\lim_{s \to 0}||h_s - h||_{C^2(\R)} = 0  $. Moreover we set $\dot{h}_s =
\frac{\partial}{\partial s}h_s$, $\dot{u}_s =  \frac{\partial}{\partial s} u_s$ and
$h_s'=\frac{\partial}{\partial\theta} h_s$. Notations $\dot{u}, \dot{h}$ mean that we evaluate the time derivatives at
$s=0$. We explicitly
point out that $ \dot{h}$ and  $ \dot{u}$ depend on $h_s$.  Finally,
for a given $h$, we define the set of functions
\begin{equation}
\label{testfunctionset}
\A(\Omega_h) := \{ w : \Omega_h \to \R^2  \mid  (h,w) \in X(0) \}.
\end{equation}
Roughly this means that $w \in \A(\Omega_h)$ if $w = 0$ outside $B_{R_0}$.

We will first write the equation for $\dot{u}$. In the following we will denote by $\tau$ the tangent unit vector to
$\Gamma_{h}$ clockwise oriented and by $\nu$ the unit normal to $\Gamma_{h}$ pointing outward the set $F_{h}$.  

\begin{proposition}
\label{prop.udot}
Let $(h,u) \in X_{\reg}(u_0)$ be such that $u$ is the elastic equilibrium associated to $h$ and $0 < h <
R_0$. Suppose  $(h_s,u_s)$ is a smooth perturbation of $(h,u)$. Then the function $\dot{u}
\in \A(\Omega_h)$ satisfies
\begin{equation}
\label{udot}
\begin{split}
\int_{\Omega_h} \C E(\dot{u}):E(w) \, dz &= \int_{\Gamma_h} \langle \vect{\dot{h}},  \nu \rangle \,  \C E(u) : E(w)  \,
d \Ha^1 \\
&= - \int_{\Gamma_h} \diver_{\tau} \left( \langle \vect{\dot{h}},  \nu \rangle\,  \C E(u) \right) \cdot w  \, d \Ha^1 ,
\end{split}
\end{equation}
for all $w \in \A(\Omega_h)$.
\end{proposition}

\begin{proof}
The proof is very similar to the one in \cite{Fusco:2009ug}.
Arguing as in \cite{cagn-mora-morini}*{Proposition
8.1} we obtain a one parameter family of 
$C^{\infty}$-diffeomorphisms  $\Phi_s( \cdot):\R^2 \setminus \{0\}\to\R^2\setminus \{0\} $ such that 
 $\Phi_0=\id$ and $\Phi_s(z)=\vect{h_s} $ on $\partial F_h$.

Suppose first that $w \in \A(\Omega_h) \cap C^{\infty}(\bar{\Omega}_h)$. We may extend $w$ outside $\Omega_h$ such that
$w \in \A(\Omega_{h_s}) \cap C^{\infty}(\bar{\Omega}_{h_s})$. Since $u_s$ is the elastic equilibrium in $\Omega_{h_s}$
we have
\[
\int_{\Omega_{h_s}} \C E(u_s) : E(w) \, dz= 0 .
\]
Differentiate this with respect to $s$ and evaluate at $s=0$ to obtain 
\[
\int_{\Omega_h} \C E(\dot{u}) : E(w) \, dz - \int_0^{2 \pi}  \dot{h}\, h\, [\C E(u) : E(w)](h\,\sigma(\theta)) \,
d\theta  = 0 .
\]
Using the area formula and notation \eqref{not.vect} we may write 
\[
\int_{\Omega_h} \C E(\dot{u}) : E(w) \, dz = \int_{\Gamma_h}  \langle \vect{\dot{h}},  \nu \rangle \, \C E(u) : E(w)  \,
d \Ha^1,
\]
where we have used the fact that the normal can be written in polar coordinates as $\nu = \frac{h \sigma + h'
\sigma^{\perp}}{\sqrt{h^2 +
h'^2}}$. The rest will follow by integration by parts and from the fact that $\C E(u)[\nu]=0$
on $\Gamma_h$ as in \eqref{euler}. 

To obtain \eqref{udot} for every $w \in \A(\Omega_h) $ one may use a standard approximation argument.
\end{proof}

\begin{remark}
\label{propertest}
Notice that the equality (\ref{udot}) clearly holds also for test functions of the form $\tilde{w}(z)= w(z) + Az + b$,
where $w \in \A(\Omega_h)$, $b \in \R^2$ and $A$ is an antisymmetric matrix. 
\end{remark}


In the next theorem we derive 
the formula for the second variation of $\F$. 
\begin{theorem}\label{2d-var:formula}
Suppose that $(h,u)$ and $(h_s,u_s)$ are as in Proposition \ref{prop.udot}. Let $\nu$ be the outer normal of
$\Gamma_h
= \partial F_h$,  $\tau$ be the tangent (positively oriented) of $\Gamma_h$ and $k$ be the  curvature of $\Gamma_h$. 
The second
variation of $\F$ at $(h,u)$ is
\begin{equation}
\label{2ndFormula}
\begin{split}
\frac{d^2}{ds^2} \F(h_s,u_s) \big|_{s=0}  =  &- \int_{\Omega_{h}} 2  Q(E(\dot{u}))  \,  dz  + \int_{\Gamma_h} | 
\partial_{\tau} \langle\vect{\dot{h}}, \nu \rangle |^2 \,  d \Ha^1  \\
&- \int_{\Gamma_h}  ( \partial_{\nu} Q(E(u)) + k^2) \, \langle \vect{\dot{h}}, \nu \rangle^2 \,  d \Ha^1   \\ 
& +  \int_{\Gamma_h}  (  Q(E(u)) - k) \, \partial_{\tau} \left( \langle  \vect{\dot{h}} , \nu \rangle \langle 
\vect{\dot{h}} , \tau \rangle \right)  \,  d \Ha^1   \\ 
&- \int_{\Gamma_h} ( Q(E(u)) - k) \, \left( \frac{\langle  \vect{\dot{h}} , \nu \rangle^2}{\langle  \vect{h} , \nu
\rangle} +  \,\langle  \vect{\ddot{h}} , \nu \rangle \right)\, d \Ha^1.
\end{split}
\end{equation}
\end{theorem}

\begin{proof}
We will treat the elastic and the perimeter part separately and write 
\[  
\F(h_s,u_s) = \int_{\Omega_{h_s}} Q(E(u_s)) \, dz +  {\Ha}^1(\Gamma_{h_s}) = \F_1(h_s,u_s) + \F_2(h_s).
\]
Since $h_s$ is smooth, we notice that $\Sigma_{h_s} = \emptyset$ and
denote by $\Phi_s $ the family of diffeomorphisms as in the proof of Proposition \ref{prop.udot}.

\noindent\textbf{1st Variation }: We start by differentiating the perimeter part $ \F_2(h_s)$.

Since ${\Ha}^1(\Gamma_{h_s}) =
\int_0^{2 \pi} \sqrt{h_s^2 + h_s'^2} \, d \theta$ we have 
\[
\frac{d}{ds}\F_2(h_s) =  \int_0^{2 \pi} \,  \frac{h_s \, \dot{h}_s + h_s' \, \dot{h}_s'}{\sqrt{h_s^2 + h_s'^2}} \, d
\theta .
\]
Integrate the second term by parts and obtain
\[
\begin{split}
\int_0^{2 \pi}  \frac{ h_s' \, \dot{h}_s'}{\sqrt{h_s^2 + h_s'^2}} \, d \theta = - \int_0^{2 \pi} \, \left(
\frac{h_s''}{\sqrt{h_s^2 + h_s'^2}} - \frac{ h_s(h_s')^2 + (h_s')^2 h_s''}{(h_s^2 +h_s'^2)^{3/2}} \right) \, \dot{h}_s
\,  d \theta .
\end{split}
\]
Then we have
\begin{equation}
\label{1stVarPer}
\begin{split}
\frac{d}{ds}\F_2(h_s) &=  \int_0^{2 \pi} h_s \, \dot{h}_s \left( \frac{ h_s^2 + 2 h_s'^2 - h_s h_s''}{(h_s^2 +
h_s'^2)^{3/2}} \right)   \, d \theta  = \int_0^{2 \pi} \,  h_s \, \dot{h}_s  \, k_s( h_s \sigma) \, d \theta, \\
&=\int_{\Gamma_{h_s}} \langle \vect{\dot{h}_s}, \nu_{h_s} \rangle \, k_s \, d \Ha^1.
\end{split}
\end{equation}
where $k_s =   \frac{ h_s^2 + 2 h_s'^2 - h_s h_s''}{(h_s^2 + h_s'^2)^{3/2}}$ is the curvature of $\Gamma_{h_s}$ in
polar coordinates.
 
Let us now treat the elastic part $ \F_1(h_s, u_s) $. Differentiate it with respect to $s$ and get, as in the proof of
Proposition \ref{prop.udot}, 
\[
\frac{d}{ds} \F_1(h_s, u_s) = \int_{\Omega_{h_s}}  \C E(\dot{u}_s) : E(u_s) \, dz - \int_0^{2 \pi}  \dot{h}_s \, h_s \,
Q(E(u_s))(h_s\sigma)  \, d \theta.
\]
The first term disappears since $u_s$ satisfies the Euler-Lagrange equations \eqref{euler} and $\dot{u}_s \in
\A(\Omega_{h_s})$. 
Hence, we are only left with
\begin{equation}
\label{1stVarEla}
\frac{d}{ds} \F_1(h_s, u_s) = - \int_0^{2 \pi}\dot{h}_s\, h_s\, Q(E(u_s))(h_s\sigma)  \, d \theta = -
\int_{\Gamma_{h_s}} \langle \vect{\dot{h}_s}, \nu_{h_s} \rangle \,  Q(E(u_s))  \, d \Ha^1.
\end{equation}
Combining \eqref{1stVarPer} and \eqref{1stVarEla} gives the first variation of $\F$.

\noindent\textbf{2nd Variation }: We will divide the
proof in
two steps.

\noindent\textit{{Step 1}}: As in \cite{Fusco:2009ug}, we begin by making a couple of general observations. 

Let $d$ be the signed distance function from $\Gamma_{h}$, i.e., 
\[d(z):=\begin{cases}
           - \dist(z,\Gamma_{h}) &\text{\ if\ }z \in F_{h},\\
           \dist(z,\Gamma_{h}) &\text{\ if\ }z \not\in F_{h}.
          \end{cases}\]
Since the boundary $\Gamma_h$ is a graph of a $C^ {\infty}$-function, $d$ is $C^1$ function in a
small tubular neighbourhood of $\Gamma_{h_t}$. Setting $\nu(z):=\nabla d(z)$ and $k(z):=(\diver
\nu)(z)$, we observe that on $\Gamma_{h}$,  $\nu$ is the outer normal to $\Gamma_{h}$ and $k$ is the curvature of
$\Gamma_{h}$. 

First we claim that 
\begin{equation}
\label{nu_curvature}
\partial_{\nu} k= -k^2, \quad \text{\ on\ } \Gamma_{h}.
\end{equation}

Differentiating the identity $\abs{\nu}=1$ with respect to $\nu$ yields $D\nu[\nu]=0$. This shows that
\begin{equation}
\label{tau_tensor_tau}
 D\nu=D_\tau\nu=k\tau\otimes\tau\quad\text{\ and\ }\quad \diver\nu=\diver_\tau\nu,\quad\text{\ on\ } \Gamma_h.
\end{equation}
Differentiating the identity $D\nu[\nu]=0$ yields
$\sum_{j=1}^2(\partial^2_{jk}\nu_i\nu_j+\partial_j\nu_i\partial_k\nu_j)=0$ for $k,i=1,2$. Hence we have
\[
\left(\partial_\nu(D\nu)\right)_{ik}=\sum_{j=1}^2
\partial^2_{jk}\nu_i\nu_j=-\sum_{j=1}^2\partial_j\nu_i\partial_k\nu_j=-\left((D\nu)^2\right)_{ik}
\]
for $i,k=1,2$. Using the previous identity we obtain
\[
\partial_\nu k=  \text{Trace} \left( \partial_{\nu}(D\nu)\right) =   -\text{Trace} \left( (D\nu)^2\right)=  -  k^2
\qquad \text{on} \, \, \Gamma_h, 
\]
where the last equality follows from \eqref{tau_tensor_tau}. Hence we have \eqref{nu_curvature}.

Next we claim that 
\begin{equation}
\label{nu_dot_tau}
\langle \dot{\nu} , \tau \rangle =- \partial_{\tau} \langle \vect{\dot{h}} , \nu \rangle,  \quad \text{\ on\ }
\Gamma_{h}.
\end{equation}
Recall that $\Phi_s( z):\R^2 \setminus \{0\}\to\R^2\setminus \{0\} $ is a one-parameter family of 
$C^{\infty}$-diffeomorphisms such that $\Phi_s(z)=\vect{h_s} $ on $\Gamma_h$ and $\Phi_0=\id$. Notice that we have
\begin{equation}
\label{nu_dot_tau2}
\langle \dot{\Phi} , \nu \rangle=  \langle \vect{\dot{h}} , \nu \rangle ,\quad \text{\ on\ } \Gamma_{h}.
\end{equation}
Differentiating $D\Phi^{-T}_s D\Phi^{T}_s[\nu]=\nu$ and calculating at $s = 0$ gives
$D\dot{\Phi}^{-T}[\nu]=-D\dot{\Phi}^T[\nu]$. Differentiate
the identity 
\[
\nu_s \circ\Phi_s=\frac{D\Phi^{-T}_s[\nu]}{\abs{D\Phi^{-T}_s[\nu]}}
\]
 with respect to $s$, evaluate at $s=0$ and use the previous identity to obtain
\begin{equation}
\label{nu_dot_tau-1}
 \dot{\nu}+D\nu[\dot{\Phi}]=-D\dot{\Phi}^T[\nu]+\langle D\dot{\Phi}^T[\nu] , \nu \rangle \,\nu \,,\quad \text{\ on\ }
\Gamma_h.
\end{equation}
By \eqref{tau_tensor_tau} we have $D\nu=D_\tau \nu^T$ on $\Gamma_h$. Therefore, multiplying \eqref{nu_dot_tau-1} 
 by $\tau$ we obtain
\[
\begin{split}
\langle \dot{\nu},\tau \rangle & =  - \langle D \dot{\Phi}^T [\nu],\tau \rangle - \langle D\nu [\dot{\Phi}],\tau\rangle
\\
&=
-\langle D \dot{\Phi}^T [\nu],\tau \rangle - \langle D\nu^T [\dot{\Phi}],\tau \rangle\\
 &= \langle \, (-D\, \langle \dot{\Phi} , \nu \rangle) \,, \tau \rangle=-\partial_{\tau}  \langle \vect{\dot{h}} , \nu
\rangle \quad\text{\ on\ } \Gamma_h
\end{split}
\] 
and \eqref{nu_dot_tau} is proven.

\noindent\textit{{Step 2}}: Let us  start with the perimeter part and differentiate \eqref{1stVarPer}
\[
\begin{split}
\frac{d^2}{ds^2}\F_2(h_s) \big|_{s=0} &=   \overbrace{\int_0^{2 \pi} h\, \dot{h}\, \dot{ k}( h\sigma) \, d \theta}^{A} +
\overbrace{\int_0^{2 \pi}   h \, \dot{h}^2  \,  \partial_{\sigma} k( h \sigma) \, d \theta}^{B} \\
&\quad+ \int_0^{2 \pi} \dot{h}^2 \, k( h\sigma) \, d \theta +  \int_0^{2 \pi} h\, \ddot{h}\, k( h\sigma) \, d \theta.
\end{split}
\]
For the term $A$ we have that 
\[
\begin{split}
A & = \int_0^{2 \pi} h\, \dot{h}\, \dot{ k}( h\sigma) \, d \theta = \int_{\Gamma_h}  \langle \vect{\dot{h}} , \nu
\rangle  \,\dot{ k} \,  d \Ha^1  = \int_{\Gamma_{h}}  \langle \vect{\dot{h}} , \nu \rangle  \, \diver_{\tau} \dot{\nu}
\,  d \Ha^1  \\
&= - \int_{\Gamma_h} \langle \dot{\nu},  \tau \rangle \, \partial_{\tau}  \langle \vect{\dot{h}} , \nu \rangle  \,  d
\Ha^1 =  \int_{\Gamma_{h}} | \partial_{\tau}  \langle \vect{\dot{h}} , \nu \rangle |^2 \,  d \Ha^1,
\end{split}
\]
where we have used  \eqref{nu_dot_tau}. For the term $B$, noticing that
\[\partial_{\sigma}k = \frac{h}{\sqrt{h^2 + h'^2}} \, \partial_{\nu} k -   \frac{h'}{\sqrt{h^2 + h'^2}} \,
\partial_{\tau} k \quad \text{ and } \quad\tau = \frac{h
\sigma^{\perp} - h' \sigma}{\sqrt{h^2 + h'^2}}\,, \]
we may write 
\[
\begin{split}
B &= \int_0^{2 \pi}   h \, \dot{h}^2  \,  \partial_{\sigma} k( h \sigma) \, d \theta =  \int_{\Gamma_h}  \langle
\vect{\dot{h}} , \nu \rangle^2 \,  \partial_{\nu} k \,  d \Ha^1 + \int_{\Gamma_h}  \langle \vect{\dot{h}}, \nu \rangle 
\langle \vect{\dot{h}}, \tau \rangle  \,  \partial_{\tau} k \,  d \Ha^1 \\
&=  - \int_{\Gamma_h}  \langle \vect{\dot{h}} , \nu \rangle^2 \,    k^2 \,  d \Ha^1 - \int_{\Gamma_h}   k \,
\partial_{\tau}\left(   \langle \vect{\dot{h}}, \nu \rangle  \langle \vect{\dot{h}}, \tau \rangle  \right)\, d \Ha^1,
\end{split}
\]
where we have used \eqref{nu_curvature} and integration by parts. Hence, we have  
\begin{equation}
\label{2ndVarPer}
\begin{split}
\frac{d^2}{ds^2}\F_2(h_s)  \big|_{s=0} &=  \int_{\Gamma_h}  | \partial_{\tau}  \langle \vect{\dot{h}} , \nu \rangle |^2 
\,  d \Ha^1  - \int_{\Gamma_h} \,  \langle \vect{\dot{h}} , \nu \rangle^2 \,   k^2 \,  d \Ha^1 \\
&\quad - \int_{\Gamma_h}   k \, \partial_{\tau}\left(   \langle \vect{\dot{h}}, \nu \rangle  \langle \vect{\dot{h}},
\tau
\rangle  \right)\,  d \Ha^1 +\int_{\Gamma_h} k \, \frac{\langle  \vect{\dot{h}} , \nu \rangle^2}{\langle  \vect{h} , \nu
\rangle}\,  d \Ha^1 \, +   \int_{\Gamma_h} k \,  \langle \vect{\ddot{h}} , \nu \rangle \, d \Ha^1.
\end{split}
\end{equation}

We are left with the elastic part. Differentiate \eqref{1stVarEla} to obtain 
\[
\begin{split}
\frac{d^2}{ds^2} \F_1(h_s, u_s) \big|_{s=0} =  &- \int_0^{2 \pi} \C E(\dot{u}):E(u) \, h \, \dot{h} \, d \theta - 
\int_0^{2 \pi} \partial_{\sigma} Q (E(u)) \, h \, \dot{h}^2 \, d \theta \\
&-  \int_0^{2 \pi} Q(E(u)) \, ( \dot{h}^2 + h \ddot{h}) \, d \theta.
\end{split}
\]
Since $\dot{u} \in \A(\Omega_h)$, we may rewrite the first term using \eqref{udot} as follows
\[
\begin{split}
  \int_0^{2 \pi}  \C E(\dot{u}):E(u) \, h \, \dot{h}   \, d \theta &=   \int_{\Gamma_h} \langle  \vect{\dot{h}} , \nu
\rangle \, \C E(u): E(\dot{u})\,  d \Ha^1 \\
		&=   \int_{\Omega_h} 2  Q (E(\dot{u}))  \,  dz.
\end{split}
\]

For the second term, noticing that 
\[\partial_{\sigma} Q(E(u)) =
\frac{h}{\sqrt{h^2 + h'^2}} \, \partial_{\nu} Q(E(u)) -   \frac{h'}{\sqrt{h^2 + h'^2}} \, \partial_{\tau} Q(E(u))\] and
using integration by parts, we get 
\[
\begin{split}
  \int_0^{2 \pi} \partial_{\sigma} Q (E(u)) \, h \, \dot{h}^2 \, d \theta &= \int_{\Gamma_h}  \partial_{\nu} Q(E(u)) \, 
\langle  \vect{\dot{h}} , \nu \rangle^2 \,  d \Ha^1  + \int_{\Gamma_h}  \partial_{\tau} Q(E(u)) \, \left( \langle 
\vect{\dot{h}} , \nu \rangle \langle  \vect{\dot{h}} , \tau \rangle \right)  \,  d \Ha^1 \\
&= \int_{\Gamma_h}  \partial_{\nu} Q(E(u)) \,  \langle  \vect{\dot{h}} , \nu \rangle^2 \,  d \Ha^1  - \int_{\Gamma_h}
Q(E(u)) \,  \partial_{\tau} \left( \langle  \vect{\dot{h}} , \nu \rangle \langle  \vect{\dot{h}} , \tau \rangle \right) 
\,  d \Ha^1\,.
\end{split}
\]

Finally we have that
\begin{equation}
\label{2ndVarEla}
\begin{split}
\frac{d^2}{ds^2} \F_1(h_s, u_s) \big|_{s=0} = &- \int_{\Omega_h} 2  Q(E(\dot{u}))  \,  dz - \int_{\Gamma_h} 
\partial_{\nu} Q(E(u)) \,  \langle  \vect{\dot{h}} , \nu \rangle^2 \,  d \Ha^1   \\
&+ \int_{\Gamma_h}  Q(E(u)) \,\partial_{\tau} \left( \langle  \vect{\dot{h}} , \nu \rangle \langle  \vect{\dot{h}} ,
\tau \rangle \right)\,  d \Ha^1 - \int_{\Gamma_h} Q(E(u)) \, \frac{\langle  \vect{\dot{h}} , \nu \rangle^2}{\langle 
\vect{h} , \nu \rangle}  \, d \Ha^1\\
 &- \int_{\Gamma_h} Q(E(u)) \,\langle  \vect{\ddot{h}} , \nu \rangle   \, d \Ha^1.
\end{split}
\end{equation}
Combining \eqref{2ndVarEla} with \eqref{2ndVarPer} yields the formula \eqref{2ndFormula}.

\end{proof}

In the formula \eqref{2ndFormula} we considered any smooth perturbation $h_s$ of $h$. However, in order to be
admissible for our minimization problem, 
a perturbation $h_s$ has to satisfy the volume constraint
$|F_{h_s}| = |F_h|$, or equivalently  
\begin{equation}
\label{volumeconstrain}
\int_0^{2 \pi} h_s^2 \, d \theta = \int_0^{2 \pi} h^2 \, d \theta \qquad \text{for all }\,\, s >0.
\end{equation}

\begin{remark}
\label{term-vanish}
If $(h,u) \in X_{\reg}(u_0)$ is a critical pair and the perturbation $(h_s)$ satisfies the volume constraint
\eqref{volumeconstrain}, then the last two terms in \eqref{2ndFormula}
vanish. 
Indeed one term vanishes because the term  
$Q(E(u)) - k$ is constant on $\Gamma_{h}$ by  \eqref{euler}. The second one vanishes since differentiating two times the
volume constraint \eqref{volumeconstrain} with respect to $s$ we
obtain
\[ \int_{\Gamma_{h}}  \frac{\langle  \vect{\dot{h}} , \nu \rangle^2}{\langle  \vect{h}
, \nu \rangle} + \langle  \vect{\ddot{h}} , \nu \rangle  \, d \Ha^1 =0\,.\]
\end{remark}

Motivated by the previous observation,
for any $\psi \in
H_{\sharp}^1(\R)$ satisfying
\begin{equation}
\label{pertur}
\int_0^{2 \pi} h \, \psi \, d \theta = 0\,,
\end{equation}
we define the  quadratic form associated to a regular critical
pair $(h,u)$
\begin{equation}
\label{biliear}
\begin{split}
\partial^2 \F(h,u)[\psi] := &- \int_{\Omega_{h}} 2  Q(E(u_{\psi}))  \,  dz + \int_{\Gamma_h} | \partial_{\tau}  \langle 
\vect{\psi} , \nu \rangle |^2 \,  d \Ha^1 \\
&- \int_{\Gamma_h} ( \partial_{\nu} Q(E(u_\psi))  +   k^2 )\,  \langle  \vect{\psi} , \nu \rangle^2 \,  d \Ha^1\,,
\end{split}
\end{equation}
where $u_{\psi} \in \A(\Omega_h)$ is the unique solution to 
\begin{equation}
\label{u_psi}
\int_{\Omega_h} \C E(u_{\psi}):E(w) \, dz=  - \int_{\Gamma_h} \diver_{\tau} \left( \langle \vect{\psi},  \nu \rangle\, 
\C E(u) \right) \cdot w \, d \Ha^1\,,\qquad \forall w \in \A(\Omega_h)\,.
\end{equation}

We define now what we mean by the second variation of $\F$ being positive at a
critical pair.
\begin{definition}
\label{pos.variat}
Suppose that $(h,u) \in X_{\reg}(u_0)$ is a critical pair. The functional \eqref{functional-relaxed} has \emph{positive
second variation} at $(h,u)$ if
\[
\partial^2 \F(h,u)[\psi] > 0
\]   
for all $\psi \in H_{\sharp}^1(\R)$ such that $\psi \neq 0$ and satisfies \eqref{pertur}.
\end{definition}

We point out that if the second variation is positive at a critical point $(h,u)$, then the formula \eqref{2ndFormula}
and Remark \ref{term-vanish} imply that for every smooth perturbation $h_s$ of $h$ satisfying the volume
constraint $\frac{d^2}{ds^2} \F(h_s,u_s) \big|_{s=0}>0$.

At the end of the section we prove the following compactness result.
\begin{lemma}
\label{easylemma}
Suppose that a critical pair $(h,u) \in X_{\reg}(u_0)$ is a point of positive second variation and $0 < h
< R_ 0$.
Then there exists $c_0 >0$  such that
\[
\partial^2 \F (h,u)[ \psi] \geq c_0 || \langle \vect{\psi}, \nu \rangle ||_{H^1(\Gamma_h)}^2,
\]
for every $\psi \in H_{\sharp}^1(\R)$ satisfying \eqref{pertur}. 
\end{lemma}
\begin{proof}
First we notice that the condition \eqref{pertur} can be written using the notation \eqref{not.vect} as
\begin{equation}
\label{orthogonalII}
\int_{\Gamma_h}   \langle  \vect{\psi} , \nu \rangle \, d \Ha^1 = 0.
\end{equation}
Using the Sobolev-Poincaré inequality $||  \langle  \vect{\psi} , \nu \rangle||_{L^2(\Gamma_h)} \leq C
||\partial_ {\tau}   \langle  \vect{\psi} , \nu \rangle||_{L^2(\Gamma_h)}$ and  \eqref{orthogonalII} we easily see that
it suffices to show that 
\[
c_0 := \inf \left\{\partial^2 \F (h,u)[ \psi] \mid \psi \in H_{\sharp}^1(\R) \text{ satisfying \eqref{pertur}},
\,\, \int_{\Gamma_h} | \partial_{\tau}  \langle  \vect{\psi} , \nu \rangle |^2 \,  d \Ha^1= 1 \right\} > 0.
\]

Choose a sequence $(\psi_n)$ such that $\psi_n$ are smooth, satisfy \eqref{pertur}, $ \int_{\Gamma_h} | \partial_{\tau} 
\langle  \vect{\psi_n} , \nu
\rangle |^2 \,  d \Ha^1= 1$ and 
\[
\partial^2 \F (h,u)[\psi_n] \to c_0. 
\] 
By restricting to a subsequence, we may assume that $\langle  \vect{\psi_n} , \nu \rangle \rightharpoonup f$  weakly in
$H^1(\Gamma_h)$. By defining
\[
\psi(\theta) := \frac{f \left(h(\theta) \sigma(\theta) \right)}{\langle \sigma, \nu \rangle} = \frac{f \left(h(\theta)
\sigma(\theta) \right)}{h(\theta)} \, \sqrt{h^2(\theta) + h'^2(\theta)} 
\]
we see that $f = \langle  \vect{\psi}, \nu \rangle$, for some $\psi  \in H_{\sharp}^1(\R)$. Moreover since 
$\int_{\Gamma_h} f\, d \Ha^1 = 0$, the function $\psi$ satisfies  \eqref{pertur}.

Next we prove that $\F(h,u)$ has the following lower semicontinuity property
\begin{equation}
\label{lowersemi}
\lim_{n \to \infty}\partial^2 \F (h,u)[\psi_n] \geq \partial^2 \F (h,u)[ \psi].
\end{equation}
Indeed, since $\langle  \vect{\psi_n} , \nu \rangle \rightharpoonup \langle  \vect{\psi} , \nu \rangle$ weakly in
 $H^1(\Gamma_h)$ then $\langle  \vect{\psi_n} , \nu \rangle \to \langle  \vect{\psi} , \nu \rangle$ strongly in
$L^2(\Gamma_h)$. Therefore we only need to check the convergence of the first term in \eqref{biliear}.

First of all, the smoothness of $\psi_n$ implies that $u_{\psi_n}$ is smooth. 
Consider the domain $\tilde{\Omega}_h = B_{2R_0} \setminus F_h$ and  the map $w_n(z)= u_{\psi_n}(z) + A_n z + b_n$,
where $A_n$ is an antisymmetric matrix and $b_n\in\R^2$ is chosen  such that $\int_{\tilde{\Omega}}w_n \, dz = 0$. Notice that   $w_n \in H^1(\tilde{\Omega}_h)$ and by Sobolev-Poincaré inequality it holds $\norm{w_n}_{L^2(\tilde{\Omega}_h)} \leq  C \norm{D w_n}_{L^2(\tilde{\Omega}_h)}$. By choosing $A_n$ such that $\int_{\tilde{\Omega}_h} Dw_n  \, dz = \int_{\tilde{\Omega}_h} Dw_n^T \, dz$ we have by Korn's inequality
(Theorem \ref{kornann}) that $ ||Dw_n||_{L^2(\tilde{\Omega}_h)} \leq \,  C \, ||E(w_n)||_{L^2(\tilde{\Omega}_h)} $.
Moreover,
since $u_{\psi_n} \equiv 0$ outside $B_{R_0}$, we have $ ||E(w_n)||_{L^2(\tilde{\Omega}_h)} =
||E(u_{\psi_n})||_{L^2(\Omega_h)}$.
By Remark \ref{propertest} we may use $w_n$ as a test function in \eqref{u_psi} and using Hölder's inequality and
the trace theorem we get 
\begin{equation} \label{WeNeedThis}
\begin{split}
\int_{\Omega_h}2 Q( E(u_{\psi_n})) \, dz &= - \int_{\Gamma_h} \diver_{\tau}
\left( \langle \vect{\psi_n},  \nu \rangle\,  \C E(u) \right) \cdot w_n  \, d \Ha^1 \\
& \leq \norm{\langle \vect{\psi_n},  \nu \rangle\,  \C E(u)}_{H^1(\Gamma_h)}
\norm{w_n}_{L^2(\Gamma_h)}\\
&\leq C \norm{\langle \vect{\psi_n},  \nu \rangle\,  \C E(u)}_{H^1(\Omega_h)}
\norm{D w_n}_{L^2(\tilde{\Omega}_h)}\\
&\leq C \norm{\langle \vect{\psi_n},  \nu \rangle\,  \C E(u)}_{H^1(\Omega_h)}
\norm{E( u_{\psi_n})}_{L^2({\Omega}_h)}\,.\\
\end{split}
\end{equation}
Therefore
\[
 \norm{D w_n}_{L^2({\tilde{\Omega}}_h)}\leq C \norm{E(u_{\psi_n})}_{L^2(\Omega_h)}\leq C\,.
 \]

However, since  $u_{\psi_n} \equiv 0$ outside $B_{R_0}$ we get 
\[
|B_{2R_0} \setminus B_{R_0}| \, |A_n|^2 =  \int_{B_{2R_0} \setminus B_{R_0}} |D w_n|^2 \, dz \leq C .
\]
This implies that the matrices $A_n$ are bounded and therefore $||Du_{\psi_n}||_{L^2(\Omega_h)} \leq \,  C$. 

By  the compactness of the trace operator we now have that $u_{\psi_n} \to u_{\psi}$ in $L^2(\Gamma_h)$ up to a
subsequence. Use $u_{\psi_n}$ as a test function in \eqref{u_psi} to obtain
\[
\begin{split}
\lim_{n \to \infty} \int_{\Omega_h}2 Q( E(u_{\psi_n})) \, dz &= - \lim_{n \to \infty}  \int_{\Gamma_h} \diver_{\tau}
\left( \langle \vect{\psi_n},  \nu \rangle\,  \C E(u) \right) \cdot u_{\psi_n}  \, d \Ha^1 \\
&=- \int_{\Gamma_h} \diver_{\tau} \left( \langle \vect{\psi},  \nu \rangle\,  \C E(u) \right) \cdot u_{\psi}  \, d \Ha^1
\\
&= \int_{\Omega_h}2 Q( E(u_{\psi})) \, dz.
\end{split}
\]
This proves \eqref{lowersemi}.

The claim now follows since if $\psi \neq 0$, the lower semicontinuity  \eqref{lowersemi} implies
\[
c_0 = \lim_{n \to \infty}\partial^2 \F (h,u)[\psi_n] \geq \partial^2 \F (h,u)[ \psi] > 0.
\]
On the other hand if $\psi \equiv 0$ then the constraint $   \int_{\Gamma_h} | \partial_{\tau}  \langle  \vect{\psi_n} ,
\nu \rangle |^2 \,  d \Ha^1= 1$ yields 
\[
c_0 = \lim_{n \to \infty}\partial^2 \F (h,u)[ \psi_n] = 1.
\]
\end{proof}

\section{\texorpdfstring{$ C^{1, 1}$-local minimality}{C\^{}\{1,1\}-local minimality}}\label{Sec:4}

In this section we perform a second order analysis of the functional \eqref{functional-relaxed} with respect to
$C^{1,1}$-topology in the spirit of \cites{dambrine-pierre}. The main result is  Proposition \ref{H2-localmin} where it
is shown that a critical point $(h,u) \in X_{\reg}(u_0)$ with positive second variation is a strict local minimizer in
the $C^{1,1}$-topology, and that the functional satisfies a growth estimate. We point
out that, according to  Lemma \ref{easylemma}, the second variation at $(h,u)$ is coercive  with respect to a norm which
is weaker than the $C^{1,1}$-norm. Therefore the local minimality does not follow directly from Lemma \ref{easylemma}.
The idea is to prove a coercivity bound in a whole $C^{1,1}$-neighborhood of the critical point, which is carried out in Lemma
\ref{2ndpostive}. The main difficulty is to control the bulk energy, which will be done by using regularity theory for
linear elliptic systems. We prove the main result first without worrying about the technicalities. All the
technical lemmata are proven later in the section.

\begin{proposition}
\label{H2-localmin}
Suppose that the critical pair $(h,u) \in X_{\reg}(u_0)$ is a point of positive second variation such that $0
< h < R_ 0 $. There exists $\delta>0$ such that for any admissible pair $(g,v) \in X(u_0)$ with $g
\in C^{1,1}_\sharp(\R)$, $||g||_{L^2([0, 2 \pi))}=||h||_{L^2([0, 2 \pi))} $ and $ ||h-g||_{C^{1,1}(\R)} 
\leq \delta$ we have
\[
\F(g,v) \geq \F(h,u) + c_1 ||h- g||_{L^2([0, 2 \pi))}^2.
\]
\end{proposition}

\begin{proof}
Assume first that $g \in C_{\sharp}^{\infty}(\R)$ and  $ ||h-g||_{C^{2}(\R)}  \leq \delta$. By scaling we may assume
that $||h||_{L^2([0, 2 \pi))} = \left( \int_0^{2 \pi} h^2 d \theta \right)^{\frac{1}{2}} = 1$. We define 
\[
g_t:= \frac{h + t (g-h)}{||h + t (g-h)||_{L^2}}
\]
so that  $g_t$  satisfies the volume constraint, and set
\[
f(t): = \F(g_t, v_t)\,,
\]
where $v_t$ are the elastic equilibria associated to $g_t$. We calculate $\frac{d^2}{dt^2} \F(g_t, v_t)$ for every
$t \in [0,1)$ by
applying the formula \eqref{2ndFormula} to  $(g_t)_s= g_{t+s}$ of $g_t$ and get
\begin{equation}
\begin{split}
f''(t)= \frac{d^2}{dt^2} \F(g_t,v_t) =  &- \int_{\Omega_{g_t}} 2  Q(E(\dot{v}_t))  \,  dz  + \int_{\Gamma_{g_t}} | 
\partial_{\tau_t} \langle\vect{\dot{g}_t}, \nu_t \rangle |^2 \,  d \Ha^1  \\
&- \int_{\Gamma_{g_t}}  ( \partial_{\nu_t} Q(E(v_t)) + k_t^2) \, \langle \vect{\dot{g}_t}, \nu_t \rangle^2 \,  d \Ha^1  
\\ 
& +  \int_{\Gamma_{g_t}}  (  Q(E(v_t)) - k_t) \, \partial_{\tau_t} \left( \langle  \vect{\dot{g}_t} , \nu_t \rangle
\langle  \vect{\dot{g}_t} , \tau_t \rangle \right)  \,  d \Ha^1   \\ 
&- \int_{\Gamma_{g_t}} ( Q(E(v_t)) - k_t) \, \left( \frac{\langle  \vect{\dot{g}_t} , \nu_t \rangle^2}{\langle 
\vect{g_t} , \nu_t \rangle} +  \,\langle  \vect{\ddot{g}_t} , \nu_t \rangle \right) \, d \Ha^1.
\end{split}
\end{equation}
 Here $\nu_t$ is the outer normal, $\tau_t$ the tangent, $k_t$ the curvature of $\Gamma_{g_t}$ and $\dot{v}_t$ is
the unique solution to 
\[
\int_{\Omega_{g_t}} \C E(\dot{v}_t):E(w) \, dz = - \int_{\Gamma_{g_t}} \diver_{\tau_t} \left( \langle \vect{\dot{g}_t}, 
\nu_t \rangle\,  \C E(v_t) \right) \cdot w  \, d \Ha^1\,,\qquad \forall w \in \A(\Omega_{g_t})\,.
\]

Remark  \ref{term-vanish} and Lemma \ref{easylemma} yield 
\[
f''(0)= \frac{d^2}{dt^2} \F(g_t,v_t) \big|_{t=0} = \partial^2 \F(h,u)[\dot{g}] \geq c_0 || \langle   \vect{\dot{g}} ,
\nu \rangle  ||_{H^1(\Gamma_h)}^2. 
\]
It will be shown later in Lemma \ref{2ndpostive} that, when $\delta>0$ is chosen to be small enough, the previous
inequality implies 
\begin{equation}
\label{f''(t)}
f''(t)= \frac{d^2}{dt^2} \F(g_t,v_t)  \geq  \frac{c_0}{2} || \langle   \vect{\dot{g}_t} , \nu_t \rangle 
||_{H^1(\Gamma_{g_t})}^2\qquad \text{for all } \,\, t \in [0,1).
\end{equation}

It is now clear that $ || \langle   \vect{\dot{g}_t} , \nu_t \rangle  ||_{H^1(\Gamma_{g_t})}^2 \geq c \, || \dot{g}_t
||_{L^2([0, 2 \pi))}^2$ holds for all $t \in [0,1]$. Since  $\int_0^{2 \pi} g^2 \, d \theta = \int_0^{2 \pi} h^2 \, d
\theta$ we have  $ \int_0^{2 \pi} (h-g)^2 d
\theta= 2
\, \int_0^{2 \pi} h(h -g) \, d \theta$ and therefore
\begin{equation}
\label{l^2norm}
|| \dot{g}_t ||_{L^2([0, 2 \pi))}^2 = \frac{1}{|| h +t(g-h) ||_{L^2}^4} \left(    \int_0^{2 \pi}(h-g)^2 d \, \theta
- \frac{1}{4}
\left( \int_0^{2 \pi}(h -g)^2 \,d \theta \right)^2   \right)\,.
\end{equation}
Since  $ \int_0^{2 \pi} (h-g)^2 d \theta $ is very small we obtain from
\eqref{l^2norm} that 
\begin{equation}
\label{L^2norm}
|| \dot{g}_t ||_{L^2([0, 2 \pi))}^2 \geq \frac{1}{2}||h-g||_{L^2([0, 2 \pi))}^2.
\end{equation}

From \eqref{f''(t)} and \eqref{L^2norm} we conclude that $f''(t) \geq \tilde{c}|| h-g||_{L^2}^2 $. Since $(h,u)$ is a
critical pair we have $f'(0)=0$ and therefore
\[
\begin{split}
\F(g,v) - \F(h,u) &= f(1) - f(0) = \int_0^1 (1-t)f''(t) \, dt \\
&\geq  \tilde{c}\, ||h -g ||_{L^2([0, 2 \pi))}^2 \int_0^1 (1-t) \, dt \\
&= \frac{ \tilde{c}}{2} \,||h-g||_{L^2([0, 2 \pi))}^2,
\end{split}
\]
which proves the claim when $g$ is smooth.

When $g \in C^{1,1,}_\sharp(\R)$ the claim follows by using a standard approximation.
\end{proof}

It remains to prove \eqref{f''(t)}. The proof is based on a compactness argument and for that we
have to study the continuity of the second variation formula \eqref{2ndFormula}.
To control the boundary terms in \eqref{2ndFormula} we need fractional Sobolev spaces whose definition and basic
properties are recalled here. The function $h$ is as in Proposition \ref{H2-localmin} and $\Gamma_h$ is its graph.

\begin{definition}
For $0 < s < 1$ and $1 < p < \infty$ we define the fractional Sobolev space $W^{s,p}(\Gamma_h)$ as the set of those
functions $v \in L^p(\Gamma_h)$  for which the Gagliardo seminorm is finite, i.e.
\begin{equation}
\label{gagli}
[v]_{s,p; \,\Gamma_h} = \left( \int_{\Gamma_h}\int_{\Gamma_h} \frac{|v(z)- v(w)|^p}{|z-w|^{1+sp}} \,  d\Ha^1(w) 
d\Ha^1(z) \right)^{1/p}  < \infty.
\end{equation}
The fractional Sobolev norm is defined as $||v||_{W^{s,p}(\Gamma_h)} := ||v||_{L^p(\Gamma_h)} + [v]_{s,p; \, \Gamma_h}
$.
\end{definition}
The space $W^{-s,p}(\Gamma_h)$ is the dual space of  $W^{s,p}(\Gamma_h)$ and the dual norm of a function $v$ is defined
as
\[
||v||_{W^{-s,p}(\Gamma_h)} := \sup \left\{ \int_{\Gamma_h} vu \,  d\Ha^1(z) \mid ||u||_{W^{s,p}(\partial F_h)} \leq 1
\right\}\,. 
\]
We also use the notation $H^s(\Gamma_h)$ for $W^{s,2}(\Gamma_h)$ for $-1 < s < 1$ and the convention $W^{0,p}(\Gamma_h)
:= L^p(\Gamma_h)$. By Jensen's inequality we have the following classical embedding theorem. 
\begin{theorem}
\label{impedding}
Let $-1 \leq t \leq s \leq 1, \, q \geq p$ such that $s - 1/p \geq t - 1/q$. Then there is a constant $C$ depending on
$t,s,p,q$ and on the $C^1$-norm of $h$ such that  
\[
|| v||_{W^{t,q}(\Gamma_h)} \leq C ||v||_{W^{s,p}(\Gamma_h)}.
\]
\end{theorem}
We also have the following trace theorem.
\begin{theorem}
\label{tracethm}
If $p>1$ there exists a continuous linear operator
$T:W^{1,p}(\Omega_h) \to W^{1 - 1/p,p}(\Gamma_h)$ such that $Tv = v|_{\Gamma_h}$ whenever $v$ is continuous on
$\bar{\Omega}_h$. The norm of $T$ depends on the $C^1$-norm of $h$ and $\gamma$.
\end{theorem}

The next lemma will be used frequently.
\begin{lemma}
\label{fractional1}
Let $-1 < s < 1$ and suppose that $v$ is a smooth function on $\Gamma_h$ . Then the following hold.
\begin{itemize}
\item[(i)] If $a \in C^1(\Gamma_h)$ then 
\[
||av||_{W^{s,p}(\Gamma_h)} \leq C ||a||_{C^1(\Gamma_h)}||v||_{W^{s,p}(\Gamma_h)},
\]
where the constant $C$ depends on $p $, $s$ and the $C^1$-norm of $h$.
\item[(ii)] If $\Psi: \Gamma_h \to \Psi(\Gamma_h)$ is a $C^1$-diffeomorphism, then
\[
||v\circ \Psi^{-1} ||_{W^{s,p}(\Psi(\Gamma_h))} \leq C ||v||_{W^{s,p}(\Gamma_h)},
\]
where the constant $C$ depends on $p $, $s$ and the $C^1$-norms of $h$, $\Psi$ and $\Psi^{-1}$.
\end{itemize}
\end{lemma}

We will also need to control the regularity of the elastic equilibrium. To this aim, the
following elliptic estimate turns out to be useful, see \cite{Fusco:2009ug}*{Lemma 4.1}.
\begin{lemma}
\label{ellip1}
Suppose $(g,v) \in X(0)$ is such that $\gamma \leq  g  \leq  R_0 - \gamma $, $g \in C_{\sharp}^2(\R)$ and $v \in
\A(\Omega_g)$ satisfies
\begin{equation}
\label{pdefornormal}
\int_{\Omega_g} \C E(v) : E(w) \, dz = \int_{\Omega_g} f : E(w) \, dz \qquad \text{for every} \,\, w \in \A(\Omega_g),
\end{equation}
where $f \in C^1(\bar{\Omega}_g; \Ma^{2 \times 2})$. Then for any $p>2$ we have the following estimate 
\begin{equation}
\label{enerest1}
\begin{split}
|| E(v)||_{W^{1,p}(\Omega_g; \Ma^{2 \times 2})} &+ || \nabla \C E(v)||_{H^{-\frac{1}{2}}(\Gamma_g; \To)}  \\
&\leq C \left(|| E(v)||_{L^2(\Omega_g ; \Ma^{2 \times 2})}  + || f||_{ C^1(\bar{\Omega}_g;  \Ma^{2 \times 2} ) } 
\right),
\end{split}
\end{equation}
where $\To$ denotes the space of third order tensors and the constant $C$ depends on $\gamma$, $p$ and the $C^2$-norm of
$g$.
\end{lemma}

We are now in position to give the proof of the inequality \eqref{f''(t)}. To control the bulk energy we use techniques developed in
\cite{Fusco:2009ug}. The main difference is that we use directly elliptic regularity rather than dealing with
eigenvalues of compact operators. 
\begin{lemma}
\label{2ndpostive}
Suppose that a critical pair $(h,u) \in X(u_0)$ is a point of positive second variation with $0 <
h < R_ 0 $ and $||h ||_{L^2} =1$ . Then there exists $\delta>0$ such that for any admissible pair $(g,v) \in
X_{\reg}(u_0)$ with $||g ||_{L^2} =1$ and  $ ||h-g||_{C^2(\R)}  \leq \delta$ we have for
\[
g_t = \frac{h +t(g-h)}{|| h + t(g-h)||_{L^2}}\,,
\]
that
\begin{equation}
\label{2ndvarpositive}
\frac{d^2}{dt^2} \F (g_t,v_t)  \geq \frac{c_0}{2} || \langle \vect{\dot{g}_t}, \nu_t \rangle ||_{H^1(\Gamma_{g_t})}^2 
\qquad \text{for all  } \, t \in [0,1], 
\end{equation}
where $v_t$ is the elastic equilibrium associated to $g_t$. The constant $c_0$ is from Lemma \ref{easylemma}.
\end{lemma}

\begin{proof}
Choose $\gamma>0$ such that  $ \gamma \leq h \leq R_0 - \gamma$. Suppose that the claim is not true and there
are pairs $(g_n,v_n)  \in X_{\reg}(u_0)$  and  $t_n \in [0,1]$ with
\[
|| h-g_n ||_{C^{2}(\R)} \to 0
\]
for which the claim doesn't hold. Denoting
\[
\dot{g}_n : = \frac{\partial}{\partial t}  _{\big|_{t=t_n}} \left( \frac{h +t(g_n-h)}{|| h + t(g_n-h)||_{L^2}} \right)
\quad\text{and}  \quad   \ddot{g}_n : = \frac{\partial^2}{\partial t^2}  _{\big|_{t=t_n}} \left( \frac{h +t(g_n-h)}{|| h
+ t(g_n-h)||_{L^2}} \right)
\]
this implies
\begin{equation}
\label{contrad1}
\lim_{n \to \infty}\frac{ \F''(g_n,v_n)}{|| \langle \vect{\dot{g}_n}, \nu_n \rangle ||_{H^1(\Gamma_{g_n})}^2}
\leq\frac{c_0}{2} ,
\end{equation}
where by the notation $ \F''(g_n,v_n)$ we mean
\begin{equation}
\label{2ndvar_n}
\begin{split}
\F''(g_n,v_n)  =  &- \int_{\Omega_{g_n}} 2  Q(E(\dot{v}_n))  \,  dz  + \int_{\Gamma_{g_n}} | 
\partial_{\tau_n} \langle\vect{\dot{g}_n}, \nu_n \rangle |^2 \,  d \Ha^1  \\
&- \int_{\Gamma_{g_n}}  ( \partial_{\nu_n} Q(E(v_n)) + k_n^2) \, \langle \vect{\dot{g}_n}, \nu_n \rangle^2 \,  d \Ha^1  
\\ 
& +  \int_{\Gamma_{g_n}}  (  Q(E(v_n)) - k_n) \, \partial_{\tau_n} \left( \langle  \vect{\dot{g}_n} , \nu_n \rangle
\langle  \vect{\dot{g}_n} , \tau_n \rangle \right)  \,  d \Ha^1   \\ 
&- \int_{\Gamma_{g_n}} ( Q(E(v_n)) - k_n) \, \left( \frac{\langle  \vect{\dot{g}_n} , \nu_n \rangle^2}{\langle 
\vect{g_n} , \nu_n \rangle} +  \,\langle  \vect{\ddot{g}_n} , \nu_n \rangle  \right)\, d \Ha^1\\
&= I_1 + I_2 +   I_3 + I_4 + I_5.
\end{split}
\end{equation}
Here $\nu_n$ is the outer normal to $F_{g_n}$, $\tau_n$ and $k_n$ are the tangent vector and the curvature of
$\Gamma_{g_n}$  and $\dot{v}_n$ is the unique solution to  
\begin{equation}
\label{dotv_n}
\int_{\Omega_{g_n}} \C E(\dot{v}_n):E(w) \, dz = - \int_{\Gamma_{g_n}} \diver_{\tau_n} \left( \langle \vect{\dot{g}_n}, 
\nu_n \rangle\,  \C E(v_n) \right) \cdot w  \, d \Ha^1,\qquad \forall w \in \A(\Omega_{g_t})\,.
\end{equation}
As in Proposition \ref{prop.udot} we find $C^{\infty}$-diffeomorphisms $\Psi_n: \bar{\Omega}_h \to 
\bar{\Omega}_{g_n}$ such that $\Psi_n:\Gamma_h \to \Gamma_{g_n}$ and 
\[
||\Psi_n - id ||_{C^2(\bar{\Omega}_h; \R^2) } \leq C ||h-g_n||_{C^2(\R)}.
\]
The goal is to examine the contribution of each term in \eqref{2ndvar_n} to the limit \eqref{contrad1}. We begin by
proving that the contribution of $I_4$ and $I_5$ to \eqref{contrad1} is zero.

Notice that  the $C^2$-convergence of $g_n$ implies $ k_n\circ \Psi_n \to k $ in $  L^{\infty}(\Gamma_h)$. Moreover,
since
$v_n$ solves the first two equations in \eqref{euler} and $\sup_{n} ||g_n||_{C^2([0,2\pi))} \leq C$, we have by a
Schauder type  estimate for Lamé system, see \cite{Fusco:2009ug}, that there is $\alpha \in (0,1)$ such that 
\begin{equation} 
\label{lame}
\sup_{n}||v_n||_{C^{1,\alpha}(\bar{\Omega}_n'; \R^2)} < \infty, \qquad \text{for} \,\, \Omega_n' = B_{R_0-\gamma}
\setminus F_{g_n}.
\end{equation}

Next we prove the following elliptic estimate
\begin{equation}
\label{ImpEst}
\begin{split}
|| E(u \circ \Psi_n^{-1}) - E(v_n  ) ||_{W^{1,p}(\Omega_{g_n} ; \Ma^{2 \times 2})} +&|| \nabla \C E(u \circ \Psi_n^{-1})
-  \nabla  \C E(v_n) ||_{H^{-\frac{1}{2}}(\Gamma_{g_n}; \To)} \\ 
&\leq C ||h- g_n||_{C^2 (\R)},
\end{split}
\end{equation}
where $p>2$ and $C$ depends on $\gamma$, $p$ and the $C^2$-norms of $h$ and $u$. Indeed by the equations \eqref{euler}
satisfied by $u$ and $v_n$ and a standard change of variables we obtain
\begin{equation}
\label{simpleEq}
\int_{\Omega_{g_n}} \C ( E(u \circ \Psi_n^{-1}) - E(v_n)): E(w) \, dz = \int_{\Omega_{g_n}} f_n :E(w) \,dz
,\qquad \forall w \in \A(\Omega_{g_n}),
\end{equation}
where $f_n \in C^1(\bar{\Omega}_{g_n} ; \Ma^{2 \times 2})$ satisfies
\[
||f_n||_{C^1(\bar{\Omega}_{g_n})}  \leq C ||h- g_n||_{C^2(\R)}
\]
for $C$ depending only on the $C^2$-norm of $u$. Lemma \ref{ellip1} yields the estimate
\[
\begin{split}
|| E(u \circ \Psi_n^{-1}) - E(v_n  ) &||_{W^{1,p}(\Omega_{g_n}; \Ma^{2 \times 2})} +|| \nabla \C E(u \circ \Psi_n^{-1})
-  \nabla  \C E(v_n) ||_{H^{-\frac{1}{2}}(\Gamma_{g_n}; \To)} \\ 
&\leq C  \left(||E(u \circ \Psi_n^{-1}) - E(v_n  )||_{L^2(\Omega_{g_n} ; \Ma^{2 \times 2})} + ||h- g_n||_{C^2 (\R)}
\right).
\end{split}
\]
On the other hand,  using $w=u \circ \Psi_n^{-1} - v_n$ as a test function in \eqref{simpleEq}, we obtain
\[
||E(u \circ \Psi_n^{-1}) - E(v_n  )||_{L^2(\Omega_{g_n} ; \Ma^{2 \times 2})} \leq C ||f_n||_{L^2(\Omega_{g_n}; \Ma^{2
\times 2})}.
\]
This concludes the proof of \eqref{ImpEst}.

By the trace theorem \ref{tracethm}, Lemma \ref{fractional1} and \eqref{ImpEst} we obtain
\[
\begin{split}
|| E(v_n \circ \Psi_n ) -
E(u)||_{H^{\frac{1}{2}}(\Gamma_{h}; \Ma^{2 \times 2})}
&\leq C|| E(v_n \circ \Psi_n ) - E(u)||_{H^1(\Omega_{h} ;\Ma^{2 \times 2})} \\
&\leq C ||g_n-h||_{C^2(\R)}.
\end{split}
\]
This estimate together with \eqref{lame} implies $v_n \circ \Psi_n  \to u$ in $C^{1, \alpha}$. In particular, we have
that 
\[
\bigl( Q(E(v_n)) - k_n \bigr)\circ \Psi_n  \to Q(E(u)) - k  \equiv \lambda 
\] 
uniformly, where $\lambda$ is a Lagrange multiplier.  We may use this to estimate the term $I_5$ in \eqref{2ndvar_n}. By
explicit calculations one easily obtains that $||\langle  \vect{\ddot{g}_n} , \nu_n \rangle||_{L^1} \leq C ||\langle 
\vect{\dot{g}_n} , \nu_n \rangle
|||_{L^2}^2$ and recalling that the  functions $\dot{g}_n $ and  $\ddot{g}_n $ satisfy the volume constraint, as in
Remark \ref{term-vanish}, 
\[
  \int_{\Gamma_{g_n}} \frac{\langle  \vect{\dot{g}_n} , \nu_n \rangle^2}{\langle 
\vect{g_n} , \nu_n \rangle} +  \,\langle  \vect{\ddot{g}_n} , \nu_n \rangle \, d \Ha^1 =  0
\]
we get 
\[
\begin{split}
\int_{\Gamma_{g_n}} &\bigl( Q(E(v_n)) - k_n\bigr) \, \left( \frac{\langle  \vect{\dot{g}_n} , \nu_n \rangle^2}{\langle 
\vect{g_n} , \nu_n \rangle} +  \,\langle  \vect{\ddot{g}_n} , \nu_n \rangle  \right)\, d \Ha^1 \\
&= \int_{\Gamma_{g_n}} \bigl( Q(E(v_n)) - k_n - \lambda \bigr) \, \left( \frac{\langle  \vect{\dot{g}_n} , \nu_n
\rangle^2}{\langle 
\vect{g_n} , \nu_n \rangle} +  \,\langle  \vect{\ddot{g}_n} , \nu_n \rangle  \right)\, d \Ha^1 \\
&\leq C \, ||Q(E(v_n)) - k_n - \lambda   ||_{L^{\infty}(\Gamma_{g_n})} ||\langle  \vect{\dot{g_n}} , \nu_n \rangle 
||_{L^2(\Gamma_{g_n})}^2.
\end{split}
\]
Using the polar decomposition we have $\nu_n=\frac{g_n\sigma+g_n'\sigma^\perp}{\sqrt{g_n^2+{g_n'}^2}}$ and
$\tau_n=\frac{g_n\sigma^\perp-g_n'\sigma}{\sqrt{g_n^2+{g_n'}^2}}$.
Since $\langle  \vect{\dot{g}_n} , \tau_n \rangle(z) = \frac{\langle z , \tau_n \rangle
}{\langle z , \nu_n \rangle}
\langle  \vect{\dot{g}_n} , \nu_n \rangle(z)$ and $||\frac{\langle  z , \tau_n \rangle }{\langle  z , \nu_n
\rangle}||_{H^1(\Gamma_{g_n})} \leq C$ we have as above that
\[
\begin{split}
  \int_{\Gamma_{g_n}}  &\bigl(  Q(E(v_n)) - k_n\bigr) \, \partial_{\tau_n} \left( \langle  \vect{\dot{g}_n} , \nu_n
\rangle
\langle  \vect{\dot{g}_n} , \tau_n \rangle \right)  \,  d \Ha^1 \\
&\leq  C  \, ||Q(E(v_n)) - k_n - \lambda   ||_{L^{\infty}(\Gamma_{g_n})} ||\langle  \vect{\dot{g_n}} , \nu_n \rangle 
||_{H^1(\Gamma_{g_n})}^2.
\end{split}
\]
Hence the contribution of the terms $I_4$ and $I_5$ to the limit \eqref{contrad1} is zero.

The remaining terms $I_1 $ , $I_2$  and $I_3$ form a quadratic form. The goal is to show that

\begin{equation}
\label{point}
\lim_{n \to \infty} \frac{\ds\F''(g_n,v_n)}{\ds|| \langle \vect{\dot{g}_n}, \nu_n \rangle ||_{H^1(\Gamma_{g_n})}^2} = \lim_{n \to \infty} \frac{\partial^{2}\F(h,u)[\vect{\psi_{n}}]}{\ds|| \langle \vect{\psi_n}, \nu \rangle||_{H^1(\Gamma_h)}^2}
\end{equation}
where 
\[
\psi_n = \frac{\dot{g}_n g_n}{h}
\] 
and $u_{\psi_n}$ solves 
\begin{equation}
\label{dotu}
\int_{\Omega_{h}} \C E(u_{\psi_n}):E(w) \, dz = - \int_{\Gamma_{h}} \diver_{\tau} \left( \langle   \vect{\psi_n} , \nu
\rangle\,  \C E(u) \right) \cdot w  \, d \Ha^1,\qquad \forall w \in \A(\Omega_{h})\,.
\end{equation}
Notice that $\psi_n $  satisfies the volume constraint $\int_0^{2 \pi} h \psi_n\, d \theta = 0$ then we may use Lemma
\ref{easylemma} to conclude that 
\[
 \lim_{n \to \infty}  \frac{\partial^2 \F (h,u)[\psi_n] }{|| \langle \vect{\psi_n}, \nu \rangle||_{H^1(\Gamma_h)}^2}
\geq  c_0\,,
\]
which then contradicts \eqref{contrad1} and proves the claim.

To show \eqref{point} we will compare the contribution of each term in the quadratic form  
\[
\partial^{2}\F(h,u)[\vect{\psi_{n}}] = -\int_{\Omega_{h}} 2  Q(E(u_{\psi_n}))  \,dz +\int_{\Gamma_{h}} |
\partial_{\tau} \langle 
 \vect{\psi_n} , \nu \rangle|^2 \,  d \Ha^1 -\int_{\Gamma_{h}}  ( \partial_{\nu} Q(E(u)) + k^2) \, \langle 
\vect{\psi_n} , \nu \rangle^2 \,  d \Ha^1 
\]
with respect to the one given by $I_{1},I_{2}$ and $I_{3}$ in \eqref{2ndvar_n}.

We first point out that since $g_n \to h$ in $C^2$ we have that $\nu_n \circ \Psi_n \to \nu$ and $\tau_n \circ \Psi_n \to \tau$ in $C^1(\Gamma_h) $. Therefore from the definition of $\psi_{n}$ we get
\[
\lim_{n \to \infty} \frac{|| \langle \vect{\psi_n}, \nu \rangle||_{H^1(\Gamma_h)}}{|| \langle \vect{\dot{g}_n}, \nu_n
\rangle||_{H^1(\Gamma_{g_n})}} =1
\]
and the convergence of  $I_2$,
\[
\lim_{n \to \infty} \frac{\ds\int_{\Gamma_{h}} | \partial_{\tau} \langle   \vect{\psi_n} , \nu \rangle|^2 \,  d \Ha^1
}{\ds\int_{\Gamma_{g_n}} | \partial_{\tau_n} \langle   \vect{\dot{g}_n} , \nu_n \rangle|^2 \,  d \Ha^1 } =1.
\]

The convergence of $I_1 $ follows from the equations \eqref{dotv_n}
and \eqref{dotu}. Indeed, by using a standard change of variables, these equations yield
\begin{equation}
\label{v_ndot}
\int_{\Omega_{g_n}} \left( \C E( u_{\psi_n} \circ \Psi_n^{-1})- \C E(\dot{v}_n) \right):E(w) \, dz = \int_{\Omega_{g_n}}
( \tilde{f}_n \, E(u_{\psi_n}\circ \Psi_n^{-1} )): E(w) \, dz + \int_{\Gamma_{g_n}} d_n \cdot w \, d \Ha^1
\end{equation}
for any $ w \in \A(\Omega_{g_n})$. Here 
\[
d_n = \diver_{\tau} (  \langle \vect{\psi_n},  \nu \rangle \C E(u))  \circ \Psi_n^{-1} \, |D_{\tau_n}\Psi_n^{-1}|
-\diver_{\tau_n} ( \langle \vect{\dot{g}_n},  \nu_n \rangle \C E(v_n)  )
\]
and $\tilde{f}_n \in L^2(\Omega_{g_n}; \Ma^{2 \times 2})$. For
$\tilde{f}_n$  we have 
\begin{equation}
\label{conv_f_n}
||\tilde{f}_n ||_{L^{\infty}(\Omega_{g_n}; \Ma^{2 \times 2})} \to 0.
\end{equation}
By the estimate \eqref{ImpEst}  we get
\[
|| \nabla \C E(u \circ \Psi_n^{-1})  -  \nabla  \C E(v_n) ||_{H^{-\frac{1}{2}}(\Gamma_{g_n}; \To)} \to 0.
\]
Therefore, by  Lemma \ref{fractional1}, the choice of $\Psi_n$ and from  $||\frac{\psi_n}{\dot{g}_n} - 1||_{C^1(\R)} \to
0$ we have that
\begin{equation}
\label{conv_d_n}
||d_n ||_{H^{-\frac{1}{2}}(\Gamma_{g_n} ; \R^2)}  || \langle \vect{\dot{g}_n}, \nu_n \rangle||_{H^1(\Gamma_{g_n})}^{-1}
\to 0.
\end{equation}

Choose 
\[
w(z) = (u_{\psi_n} \circ \Psi_n^{-1}- \dot{v}_n)(z) +Az + b 
\]
as a test function in \eqref{v_ndot} where $A$ is antisymmetric and $b$ is a vector. This yields
\begin{equation}
\label{looooong}
\begin{split}
\int_{\Omega_{g_n}} Q( E(u_{\psi_n} & \circ \Psi_n^{-1} - \dot{v}_n)) \, dz\\
&\leq  C || \tilde{f}_n||_{L^{\infty}(
\Omega_{g_n};
\Ma^{2 \times 2})} || E(u_{\psi_n})||_{L^{2}( \Omega_{h};
\Ma^{2 \times 2})} ||E (u_{\psi_n} \circ \Psi_n^{-1}- \dot{v}_n )||_{L^2( \Omega_{g_n}; \Ma^{2 \times 2})}\\
&\quad +|| d_n ||_{H^{-\frac{1}{2}}(\Gamma_{g_n}; \R^2) } ||w ||_{H^{\frac{1}{2}}(\Gamma_{g_n}; \R^2)}.
\end{split}
\end{equation}
By Theorem \ref{tracethm} we get that 
\[
||w||_{H^{\frac{1}{2}}(\Gamma_{g_n}; \R^2)} \leq C||w||_{H^1(\Omega_{g_n}; \Ma^{2 \times 2})}.
\]
As in the proof of Lemma \ref{easylemma} we choose $A$  such that
\[
||w||_{H^1(\Omega_{g_n}; \Ma^{2 \times 2})} \leq C ||E (w )||_{L^2( \Omega_{g_n}; \Ma^{2 \times 2})}= C  ||E (
u_{\psi_n}
\circ \Psi_n^{-1}- \dot{v}_n )||_{L^2( \Omega_{g_n}; \Ma^{2 \times 2})},
\]
by Korn's and Poincaré's inequalities (choose $b$ accordingly). The two previous inequalities and \eqref{looooong} yield
\[
||E( u_{\psi_n} \circ \Psi_n^{-1}- \dot{v}_n )||_{L^2(\Omega_{g_n}; \Ma^{2 \times 2})}\!\leq\! C\! \left(  ||
\tilde{f}_n||_{L^{\infty}( \Omega_{g_n};
\Ma^{2 \times 2})} || E(u_{\psi_n})||_{L^{2}( \Omega_{h};
\Ma^{2 \times 2})}+ ||d_n ||_{H^{-\frac{1}{2}}(\Gamma_{g_n}; \R^2) } \right).
\]
Arguing as in \eqref{WeNeedThis} we may estimate
\[
 || E(u_{\psi_n})||_{L^{2}( \Omega_{h}; \Ma^{2 \times 2})} \leq C\, || \langle \vect{\psi_n}, \nu
\rangle||_{H^1(\Gamma_{h})} .
\]
Therefore using \eqref{conv_f_n} and \eqref{conv_d_n}  we deduce that 
\[
\frac{ ||E(  u_{\psi_n} )||_{L^2(\Omega_{h}; \Ma^{2 \times 2})}^2 - ||  E(\dot{v}_n )||_{L^2(\Omega_{g_n}; \Ma^{2 \times
2})}^2 }{ || \langle \vect{\dot{g}_n}, \nu_n \rangle||_{H^1(\Gamma_{g_n})}^2} \to 0.
\] 
This proves the convergence of $I_1$.

We are left with the term $I_3$ in \eqref{2ndvar_n}. We need to show that
\[
\bigg| \int_{\Gamma_{g_n}}  ( \partial_{\nu_n} Q(E(v_n)) + k_n^2) \, \langle  \vect{\dot{g}_n} , \nu_n \rangle^2 \,  d
\Ha^1 - \int_{\Gamma_{h}}  ( \partial_{\nu} Q(E(u)) + k^2) \, \langle  \vect{\psi_n} , \nu \rangle^2 \,  d \Ha^1 \bigg|
\,  || \langle \vect{\dot{g}_n}, \nu_n \rangle||_{H^1(\Gamma_{g_n})}^{-2} \to 0.
\]
Due to the $C^2$-convergence of $g_n$ and  the $C^1$-convergence of $\frac{\psi_n}{g_n}$ we just need to show 
\begin{equation}
\label{H-onehalf}
||  \partial_{\nu_n}Q(E(v_n)) \circ \Psi_n -  \partial_{\nu}Q(E(u)) ||_{H^{-\frac{1}{2}}(\Gamma_h)} \to 0.
\end{equation}
This will be done as  \cite{Fusco:2009ug}*{Proposition 4.5}.  
For every $\varphi \in H^{\frac{1}{2}}(\Gamma_h)$ we have  
\[
\begin{split}
&\int_{\Gamma_h} \left( \frac{\partial}{\partial x_1 }  Q(E(v_n))  \circ \Psi_n -   \frac{\partial}{\partial x_1 }
Q(E(u)) \right) \varphi \, d \Ha^1 \\
&=\int_{\Gamma_h} \left(  \C E \left( \frac{\partial v_n}{\partial x_1 } \right)  \circ \Psi_n -   \C E  \left(
\frac{\partial u}{\partial x_1 } \right)   \right) : \, (E(v_n ) \circ \Psi_n) \,  \varphi \, d \Ha^1 \\
&\quad+\int_{\Gamma_h} \C E  \left( \frac{\partial u}{\partial x_1 } \right): \, (  E( v_n)  \circ \Psi_n -   E(u))   \,
 \varphi \, d \Ha^1 \\
&\leq  || (\nabla \C E(v_n)) \circ \Psi_n  - \nabla \C E(u)||_{H^{-\frac{1}{2}}(\Gamma_h; \To)} || ( E(v_n) \circ
\Psi_n) \,  \varphi ||_{H^{\frac{1}{2}}(\Gamma_h; \Ma^{2 \times 2}))} \\
&\quad + C || E(v_n) \circ \Psi_n  - E(u)||_{L^2(\Gamma_h; \Ma^{2 \times 2})} || \varphi ||_{L^2(\Gamma_h)} 
\end{split}
\]
where the constant depends on $C^2$-norms of $u$ and $h$. Fix $p >2$. By the definition of Gagliardo seminorm,
H\"older's inequality,  Theorem \ref {impedding} and  Theorem \ref{tracethm}, we obtain
\[
\begin{split}
|| &( E(v_n) \circ \Psi_n) \, \varphi ||_{H^{\frac{1}{2}}(\Gamma_h; \Ma^{2 \times 2}))} \\
&\leq C|| ( E(v_n) \circ \Psi_n)||_{L^{\infty}(\Gamma_h; \Ma^{2 \times 2}))} || \varphi ||_{H^{\frac{1}{2}}(\Gamma_h) }
+ C || ( E(v_n) \circ \Psi_n)||_{W^{\frac{p+2}{2p}, \frac{2p}{p-2}}(\Gamma_h; \Ma^{2 \times 2}))} || \varphi
||_{L^p(\Gamma_h) } \\
&\leq C \left( || ( E(v_n) \circ \Psi_n)||_{L^{\infty}(\Gamma_h; \Ma^{2 \times 2}))} + || ( E(v_n) \circ
\Psi_n)||_{W^{1, \frac{2p}{p-2}}(\Omega_h; \Ma^{2 \times 2})}  \right)\, || \varphi ||_{H^{\frac{1}{2}}(\Gamma_h) }.
\end{split}
\]

By repeating the previous argument for $\frac{\partial}{\partial x_2 }$ we obtain by \eqref{lame} and \eqref{ImpEst}
that
\begin{equation}
\label{GradWE}
\begin{split}
||\nabla &Q(E(v_n))  \circ \Psi_n -   \nabla Q(E(u))  ||_{H^{-\frac{1}{2}}(\Gamma_h; \R^2)} \\
&\leq C \left( || (\nabla \C E(v_n)) \circ \Psi_n  - \nabla \C E(u)||_{H^{-\frac{1}{2}}(\Gamma_h; \To)}  + || E(v_n)
\circ \Psi_n  - E(u)||_{L^2(\Gamma_h; \Ma^{2 \times 2})}  \right)\\
&\leq C||g_n-h||_{C^2 (\R)}.
\end{split}
\end{equation}
Since $\nu_n \circ \Psi_n \to \nu$ in $C^1$, \eqref{GradWE}
implies \eqref{H-onehalf}. This concludes the convergence of the term $I_3$ and completes the proof.
\end{proof}

\section{Local minimality}\label{Sec:5}

This section is devoted to prove the main result of the paper, the local minimality criterion. Namely, we show that if a
critical point $(h,u)\in X_{\reg}(u_0)$ has positive second variation, then it is a strict local minimizer in the
Hausdorff distance of sets and a quantitative estimate in terms of the measure of the symmetric difference between the minimum
and a competitor holds. Due to the sharp quantitative isoperimetric inequality, the exponent 2 in  \eqref{main-equation} is optimal.

\begin{theorem}\label{TH:criterion} Suppose that $(h,u) \in X_{\reg}(u_0)$ is a critical pair for $\F$ with $0 <
h  < R_{0}$. If the second variation of $\F$ is positive at $(h,u)$, then there is $\delta>0$
such that for any 
$(g,v) \in X(u_0)$ with $|\Omega_g| = |\Omega_h|$ and $0<d_{\Ha}(\Gamma_g\cup\Sigma_g, \Gamma_h) \leq \delta$
it holds that
\begin{equation}\label{main-equation}
\F(g,v) > \F(h,u) +  c\,  \abs{\Omega_{g}\Delta \Omega_{h}} ^2,
\end{equation}
for some $c >0$.
\end{theorem}

The proof is based on a contradiction argument and follows some ideas contained in \cite{Fusco:2009ug}, \cite{Cicalese:2010uc} and
\cite{AcFuMo}. Assume, for a contradiction, that $(h_{n},u_{n})$  is a sequence satisfying
\[
\F(h_{n},u_{n}) \leq \F(h,u) + c_0 \,|\Omega_{h_n} \Delta \Omega_h|^2
\;\;\;\textrm{and}\;\;\; 0 < d_{\Ha}({\Gamma}_{h_{n}}\cup\Sigma_{h_{n}},\Gamma_{h})\leq \frac{1}{n}.
\]
The idea is to replace $(h_{n},u_{n})$ with the minimizer  $(g_{n},v_{n})$ of an auxiliary
constrained-penalized problem, and to prove that the $(g_{n},v_{n})$ are sufficiently regular to apply the
$C^{1,1}$-minimality criterion to get a contradiction. As auxiliary problem we choose 
\[
\min\left\{  \F(g,v)+\Lambda \big| |\Omega_{g}|-|\Omega_{h}|\big| +  \sqrt{(\abs{\Omega_{g}\Delta \Omega_{h}} - \eps_n
)^2+\eps_n } \;:\; (g,v)\in X(u_{0})\;,\;g\leq h+\frac{1}{n} \right\},
\]
where the second penalization term will provide the
quantitative estimate in \eqref{main-equation} and
the obstacle $g\leq h+1/n$ plays a key role in proving the regularity of $(g_{n},v_{n})$. 

The regularity proof is
divided in three steps. In Lemma \ref{lem-balltohaus} we prove that $g_n$ is Lipschitz using some geometrical
arguments. Then, in Lemma \ref{lem-more-regular}, we show that $g_n$ is a quasiminimizer for the area functional which
in turns implies its $C^{1,\alpha}$-regularity. Finally, we deduce the $C^{1,1}$-regularity in Lemma
\ref{lem-convergence}, by using the Euler-Lagrange equation for $(g_n,v_n)$. 

The following isoperimetric-type result will be used frequently in this section. The proof can be found in
\cite{AcFuMo}*{Lemma 4.1}.
\begin{lemma}\label{lemma-isoperimetric}
\ 
\begin{itemize}
 \item[(i)]Let $f \in C_{\sharp}^{\infty}(\R)$ be non-negative and let $g\in BV_{\sharp}(\R)$, then there exists a
constant $C$, depending only on $f$, such that 
\begin{equation*}
\Ha^1(\Gamma_g)-\Ha^1(\Gamma_f)\geq -C\abs{\Omega_g\Delta \Omega_f}\,.
\end{equation*}
\item[(ii)] Suppose $D$ is a set of finite perimeter. Then 
\[
P(D \cup B_r(x))- P( B_r(x))\geq  \frac{1}{r}\abs{D}\,,
\]
where $P$ stands for the perimeter.
\end{itemize}
\end{lemma}

We will also need the following property of concave functions.
\begin{lemma}
\label{conc-conv}
Suppose that $f_n \in C^1(\R)$  and $f \in C^1(\R)$ are such that $||f_n - f||_{L^{\infty}(\R)} \to 0$. If the $f_n$ are
concave then
\[
||f_n - f||_{C_{loc}^{1}(\R)} \to 0.
\] 
\end{lemma}

\begin{proof}
Let $R >0$ and fix $\eps>0$ . Since $f \in C^1(\R)$ we find $\delta>0$ such that 
\[
f(\delta+x) - f(x) \geq f'(x) \delta - \eps \delta  
\] 
for every $|x| \leq R$. On the other hand, since the $f_n$ are concave, we have
\[
\frac{f_n(\delta+x) - f_n(x)}{\delta} \leq f_n'(x) .
\]
Hence 
\[
f'(x) - f_n'(x) \leq  \frac{f(\delta+x) - f_n(\delta+x)- (f(x)-  f_n(x))}{\delta} + \eps \leq 2 \eps,
\]
when $n$ is large. The reverse inequality  $f_n'(x) - f'(x) \leq 2 \eps$  follows from a similar argument. 
\end{proof}

We begin the study of the properties of solutions of the auxiliary problem by  proving an exterior ball condition.
\begin{theorem}\label{thm:ext_ball}
Let  $h \in C^{\infty}_{\sharp}(\R)$ such that $0 < h < R_{0}$. Then for every
$c, \eps \in [0,1]$ and $n \in \N$  every solution of the problem
\begin{equation}\label{eq:min_pbl}
 \min \left\{\F(g,v)+\Lambda\bigl| \abs{\Omega_g}- \abs{\Omega_h} \bigr| + c \sqrt{(\abs{\Omega_g \Delta  \Omega_h} -
\eps )^2+\eps }: (g,v)   \in
X(u_{0}),\, g\leq
h + \frac{1}{n}\right\}\,,
\end{equation}
satisfies the following \emph{uniform exterior ball condition}:
for every $z \in \partial F_g$ and for every
$r < \min\{1/(\Lambda+1),1/\norm{\kappa_h}_\infty\}$, 
where $\kappa_h$ is the curvature of $\Gamma_h$, 
there exists $z_0$ such that  $B_r(z_0) \subset \R^2 \setminus F_g$ and $\partial B_r(z_0)\cap
\left(
\Gamma_g\cup\Sigma_g \right)=\{z\}$.
\end{theorem}

\begin{figure}[ht]
 \caption{}\label{fig:ext_ball}
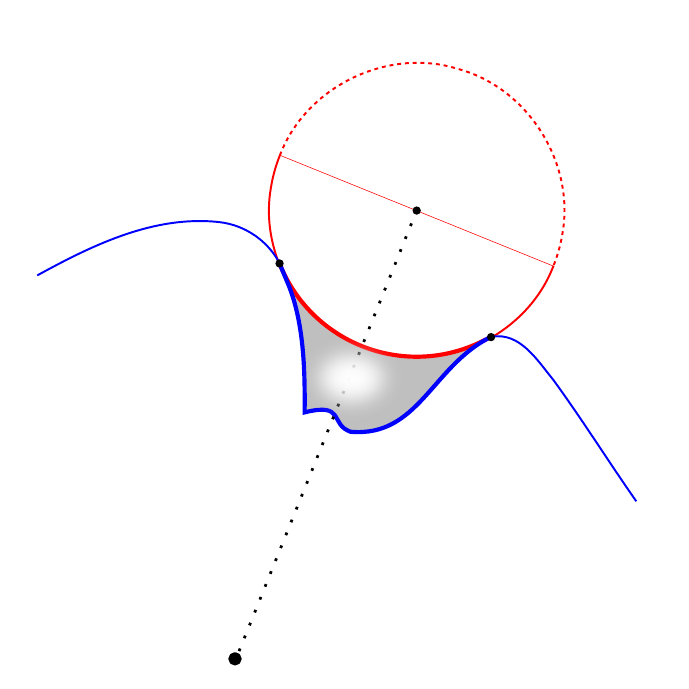
\end{figure}

\begin{proof}
The proof follows the argument from \cite{Fusco:2009ug}*{Lemma 6.7}. Recall that $\partial F_g = \Gamma_g\cup\Sigma_g$. 
Given a ball $B_r(z_0)$ define the half circle $S_r(z_0)=\partial B_r(z_0)\cap\{z\in\R^2:\langle
z-z_0,z_0\rangle <0\}.$  Assume that there exists a ball $B_r(z_0)\subset\R^2\setminus F_g$ such that $S_r(z_0)$
intersects $\Gamma_g\cup \Sigma_g$ in two different points $z_1=(\rho_1,\theta_1)$ and
$z_2=(\rho_2,\theta_2)$. When $r < 1/ \norm{\kappa_h}_\infty$  it is clear that the arc $f=f(\theta)$ of $S_r(z_0)$
connecting $z_1$ and $z_2$ satisfies $f(\theta)\leq h(\theta)+\frac{1}{n}$ for $\theta\in(\theta_1,\theta_2)$. Let $\tilde{g}$ be 
defined by $\tilde{g}=f $ for $\theta\in(\theta_1,\theta_2)$ and $\tilde{g}=g$ otherwise. Denote by $\tilde{f}$
the arc of $\Gamma_g\cup\Sigma_g$ connecting $z_1$ and $z_2$ and by $D$ the region enclosed by $f\cup\tilde{f}$, see
Figure \ref{fig:ext_ball}.

Notice that 
\begin{equation}\label{eq:trick_vol}
 \begin{split}
\sqrt{(\abs{\Omega_{\tilde{g}}\Delta \Omega_h}-\eps)^2+\eps} &-\sqrt{(\abs{\Omega_{{g}}\Delta \Omega_h}-\eps)^2+\eps}
\\ 
&\qquad=\frac{(\abs{\Omega_{\tilde{g}}\Delta \Omega_h}-\eps)^2-(\abs{\Omega_{{g}}\Delta \Omega_h}-\eps)^2}
{\sqrt{(\abs{\Omega_{\tilde{g}}\Delta \Omega_h}-\eps)^2+\eps}+\sqrt{(\abs{\Omega_{{g}}\Delta \Omega_h}-\eps)^2+\eps}}
\\ &\qquad\leq
\frac{(\abs{\Omega_{\tilde{g}}\Delta \Omega_h}+\abs{\Omega_g\Delta \Omega_h}-2\eps)(\abs{\Omega_{\tilde{g}}\Delta
\Omega_h}
-\abs{\Omega_g\Delta \Omega_h})}
{\bigl| \abs{\Omega_{\tilde{g}}\Delta \Omega_h}-\eps\bigr|+ \bigl| \abs{\Omega_{{g}}\Delta \Omega_h}-\eps\bigr|}\\
&\qquad\leq \abs{\Omega_{\tilde{g}}\Delta \Omega_g}\,.
 \end{split}
\end{equation}
Since $\Omega_{\tilde{g}}\Delta \Omega_g=D$ and $\Omega_{\tilde{g}}\subset\Omega_g$ we see that
\begin{equation}\label{eq:ext_ball}\begin{split}
  &\F(\tilde{g},v)+\Lambda \bigl|\abs{\Omega_{\tilde{g}}}-|\Omega_h|\bigr|+c
  \sqrt{(\abs{\Omega_{\tilde{g}}\Delta \Omega_h}-\eps)^2+\eps}
  \\
&\leq\F({g},v) + \Lambda \bigl|\abs{\Omega_{g}} - |\Omega_h|\bigr|+c
  \sqrt{(\abs{\Omega_{{g}}\Delta \Omega_h}-\eps)^2+\eps}+ \Ha^1(f)-\Ha^1(\tilde{f})+(\Lambda+1)\abs{D}\,.
  \end{split}\end{equation}
Moreover from  Lemma \ref{lemma-isoperimetric} we infer that 
\[
 \Ha^1(f)-\Ha^1(\tilde{f})\leq P(B_r(z_0)) -P(D \cup B_r(z_0))  \leq -\frac{1}{r}\abs{D}\,.
\]
Hence, since $r<1/(\Lambda+1)$,  the inequality \eqref{eq:ext_ball} contradicts the minimality of $(g,v)$. The
conclusion now follows arguing as \cite{chambolle}*{Lemma 2} or \cite{Fonseca:2007bj}*{Proposition 3.3, Step 2}.
\end{proof}

\begin{lemma}\label{volume-lemma}
Let $h, c, \eps$ and $n$ be as in the previous theorem. Suppose $(g,v)\in X(u_{0})$ is any minimizer of 
\eqref{eq:min_pbl}. Then there exists $\Lambda_{0}>0$, independent of $c, \eps$ and $n$, such that if $\Lambda\geq
\Lambda_{0}$ then $|\Omega_{g}|\geq |\Omega_{h}|$.
\end{lemma}
\begin{proof}
We argue by contradiction supposing that $|\Omega_{g}|<  |\Omega_{h}|$ for every $\Lambda >0$. We observe that there
exists $0<r<1$ such that,
if we define
$\Omega^r_g=B_{R_{0}}\setminus r F_g$, we have $|\Omega_{g}^{r}|=  |\Omega_{h}|$. Moreover, since
\[
|\Omega^{r}_{g}|=\pi R_{0}^{2}-\frac{r^{2}}{2} \int_{0}^{2\pi}g^{2}\,d\theta,
\]
we get
\[
r=\left( \frac{\pi R^{2}_{0}- |\Omega_{h}|}{\pi R_{0}^{2}-|\Omega_{g}|} \right)^{\frac{1}{2}} < 1.
\]
Clearly $\Omega_{g}^{r}=\Omega_{g_{r}}$ for $g_{r}(\theta)=r g(\theta)$. Define the function $v_{r}: \Omega_{g_{r}}\to
\R^{2}$ as
\[
v_{r}(z)=\begin{cases}
u_{0}\left(\frac{z}{|z|}R_{0}\right) & \textrm{if } r R_{0}\leq |z| \leq R_{0} \\
v\left( \frac{z}{r}\right) & \textrm{if } g_{r}\left( \frac{z}{|z|}\right) \leq |z| < r R_{0}.
\end{cases}
\]
Since $\Omega_{g_r}\supset \Omega_g$, we see that $\abs{\Omega_{g_r}\Delta\Omega_g}= |\Omega_{h}|-\abs{\Omega_g}$.
Using the inequality  \eqref{eq:trick_vol} we have, for $\Lambda$ sufficiently large, that
\[
\begin{split}
&\F(g_{r},v_{r})  +\Lambda \big| |\Omega_{g_{r}}|- |\Omega_{h}| \big| 
+ c \sqrt{( \abs{ \Omega_{g_r}\Delta \Omega_{h} } - \eps )^2+\eps } \\
&- \F(g,v) - \Lambda  \big| |\Omega_{g}| - |\Omega_{h}|  \big| 
- c \sqrt{( \abs{ \Omega_{g}\Delta \Omega_{h} } - \eps )^2+\eps }\\
	& \qquad\qquad\leq \int_{rR_{0}\leq |z|\leq R_{0}} Q(E(v_{r})) \, dz- \Lambda \left(  |\Omega_{h}|- |\Omega_{g}|
\right)
+c\abs{\Omega_{g_r}\Delta\Omega_g}\\
	&\qquad\qquad \leq C(1-r) - (\Lambda-1) \left( |\Omega_{h}|- |\Omega_{g}| \right)\\
	&\qquad\qquad \leq C(\abs{\Omega_h}-\abs{\Omega_g}) - (\Lambda-1) \left( |\Omega_{h}|- |\Omega_{g}| \right) < 0\,,
\end{split}
\]
which  contradicts the minimality of $(g,v)$. 
\end{proof}

In the following we study convergence properties of solutions for the constrained obstacle problem \eqref{eq:min_pbl}.
\begin{lemma}
\label{lem-balltohaus}
Let  $h$ be as in Theorem \ref{thm:ext_ball}. Assume
$g_{n}\in BV_{\sharp}(\R)$ is such that $g_{n}\leq h+ 1/n$ and it satisfies the uniform
exterior ball condition. If 
\begin{equation}\label{eq:len2len}
g_{n}\to h \text{ in } L^{1} \text{ and }
\lim_{n\to\infty}\Ha^1(\Gamma_{g_{n}}\cup\Sigma_{g_n})=\Ha^{1}(\Gamma_{h})\,,
\end{equation}
then $g_{n}\to h$ in $L^{\infty}$. Moreover, for $n$ sufficiently large,  the $g_{n}$ are uniformly Lipschitz
continuous.
\end{lemma}
\begin{proof}Here we follow an argument from \cite{Fusco:2009ug}*{Theorem 6.9, Steps 1 and 2}.
We claim that
\[
 \sup_{\R}\,\abs{g_n-h}\to 0 \text{\ as\ }n\to +\infty\,.
\]
Let us first note that $\Gamma_{g_n}\cup\Sigma_{g_n}$ is a connected compact set. Up to a
subsequence, we can assume that $\Gamma_{g_n}\cup\Sigma_{g_n}$ converges in the Hausdorff distance to
some compact connected set $K$. The continuity of $h$ and condition \eqref{eq:len2len} imply that $\Gamma_h\subset K$.
By Go\l\k{a}b's semicontinuity Theorem (see, e.g.\ \cite{ambrosio-tilli}*{Theorem 4.4.17}) and assumption
\eqref{eq:len2len} we see that
\[
 \Ha^1(\Gamma_h)\leq\Ha^1(K)\leq \lim_{n\to +\infty}\Ha^1(\Gamma_{g_n}\cup\Sigma_{g_n})=\Ha^1(\Gamma_h)\,.
\]
This implies that $\Ha^1(K\setminus\Gamma_h)=0$. Since $K$ is connected, it follows from a density lower bound  (see,
e.g.  \cite{ambrosio-tilli}*{Lemma 4.4.5}) that $K=\Gamma_h$. Now the claim follows from the definition of the Hausdorff
metric and from the continuity of $h$.

From the previous claim we see that, for $n$ sufficiently large, it holds
$\gamma\leq g_n\leq R_0 -\gamma$ for some $\gamma>0$ small, independent from $n$. Hence, since
the polar coordinates mapping is a $C^\infty$-local diffeomorphism far from the origin,
the representation in polar coordinates of
$F_{g_n}$ (still denoted by $F_{g_n}$) satisfies the uniform exterior ball condition up to changing the radius $r$ to
$\tilde{r} \in (0,1)$ by a factor depending only on $\gamma$.  Now we prove that $g_n$ are $L$-Lipschitz with $L \leq
\frac{8}{\tilde{r}} \, ||h||_{C^1(\R)}$.

We argue by contradiction and assume that there exists $\theta$ and $\theta_k \to \theta$ such that
\[
\lim_{k \to \infty}  \frac{|g_n(\theta_k) - g_n(\theta)|}{|\theta_k - \theta|} \geq  \frac{8}{\tilde{r}} \,
||h||_{C^1(\R)}
\]
and set $z=(\theta,g_n(\theta))$. Without loss of generality we may assume that the sequence $\{\theta_k\}_k\in\N$ is monotone and $g_n(\theta_k)$ is increasing. By the uniform
exterior ball condition we find a ball $B_{\tilde{r}}(z_0)\subset\R^2\setminus F_{g_n}$ such that  $\partial
B_{\tilde{r}}(z_0)\cap
(\Gamma_{g_n}\cup\Sigma_{g_n})=\{z\}$ and 
\[
z_0 =  z+\tilde{r} \left(\frac{M}{\sqrt{1+M^2}},\frac{1}{\sqrt{1+M^2}}\right), \quad \text{for} \quad M \geq
\frac{4}{\tilde{r}} \, ||h||_{C^1(\R)}
\]
Let  $z' \in \partial B_{\tilde{r}}(z_0)$ such that
\[
 z' = z_0 - \tilde{r}\left(\frac{\sqrt{ M^2 -3}}{\sqrt{1+M^2}},\frac{ 2}{\sqrt{1+M^2}}\right) .
\]
We write $z' =:z+ \tilde{r} \, (w_1,w_2)$ with 
\[
w_1 =   \frac{M- \sqrt{ M^2 -3}}{\sqrt{1+M^2}}  > 0 \quad \text{and} \quad  w_2 = \frac{-1}{\sqrt{1+M^2}} <0
\]
and since $B_{\tilde{r}}(z_0)\subset\R^2\setminus F_{g_n}$ we have $g_n(\theta+\tilde{r} w_1)\leq g_n(\theta)+\tilde{r}
w_2$. Setting $\delta_n=\sup_\R \abs{h-g_n}$ and recalling $ ||h||_{C^1(\R)} \leq M/ 4$ we get
\[
 h(\theta+\tilde{r}  w_1)\geq h( \theta)-  \frac{ \tilde{r} M}{4} w_1 \geq g_n(\theta)- \delta_n - \frac{ \tilde{r}
M}{4} w_1\,.
\]
Therefore we deduce 
\[
\begin{split}
 h(\theta+\tilde{r} w_1)-g_n(\theta+ \tilde{r} w_1) &\geq - \delta_n - \tilde{r} \left( \frac{ M}{4}  w_1- w_2 \right)
\\ 
&= - \delta_n + \frac{\tilde{r}}{\sqrt{1+M^2}} \left( 1 - \frac{M}{4} \left(M- \sqrt{ M^2 -3}\right) \right) \\
&=  - \delta_n + \frac{\tilde{r}}{\sqrt{1+M^2}} \left( 1 - \frac{3M}{4(M + \sqrt{ M^2 -3})}\right) >  \delta_n 
\end{split}
\]
where the last inequality, which holds for $n$ sufficiently large, gives a contradiction.
\end{proof}

In the next lemma we show the $C^{1, \alpha}$-regularity of the minimizer for the penalized obstacle problem. 

\begin{lemma}\label{lem-more-regular}
Let $h$ be as in Theorem \ref{thm:ext_ball} and $(g_{n},v_{n})\in X(u_{0})$ be any minimizer of the problem
\begin{equation}
\label{panalized.funct}
\min\left\{  \F(g,v)+\Lambda \big| |\Omega_{g}|-|\Omega_{h}|\big| + c \sqrt{(\abs{\Omega_{g}\Delta \Omega_{h}} -
\eps_n )^2+\eps_n } \;:\; (g,v)\in X(u_{0})\;,\;g\leq h+\frac{1}{n} \right\},
\end{equation}
where  $c \in [0,1]$ and $\eps_n \to 0$. Assume also that $g_{n}\to h$ in $L^{1}$ and that
\[
\lim_{n\to\infty}\Ha^{1}({\Gamma}_{g_{n}}\cup \Sigma_{g_{n}})=\Ha^{1}(\Gamma_{h}) \;\; \textrm{ and } \;\;
\sup_{n}\int_{\Omega_{g_{n}}}Q(E(v_{n}))dz < +\infty.
\]
Then for all $\alpha \in \left(0,\frac{1}{2}\right)$ and for $n$ large enough $g_{n}\in C^{1,\alpha}(\R)$, the
sequence $\{\nabla v_{n}\}$ is equibounded in $C^{0,\alpha}(\overline{\Omega}_{g_{n}};\mathbb{M}^{2\times 2})$, and
$g_{n}\to h$
in $C^{1,\alpha}(\R)$.
\end{lemma}
\begin{proof}
From Lemma \ref{lem-balltohaus} we infer that $g_{n} $ is sufficiently regular to ensure a decay estimate for $\nabla
v_{n}$. Indeed, for $z_{0}\in \Gamma_{g_{n}}$ there exist $c_{n}>0$, a radius $r_{n}$ and an exponent $\alpha_{n}\in
(0,1/2)$ such that
\[
\int_{B_{r}(z_{0})\cap \Omega_{g_{n}}} |\nabla v_{n}|^{2}\leq c_{n}r^{1+2\alpha_{n}},
\]
for every $r < r_n$. This follows from the fact that $v_{n}$ minimizes  the elastic energy in
$\Omega_{g_{n}}$ and the boundary $\Gamma_{g_n}$ is Lipschitz, see Theorem 3.13 in
\cite{Fonseca:2007bj}.

Since $g_{n}$ is Lipschitz, we may extend $v_{n}$ in $B_{r}(z_{0})$ such that
\begin{equation}\label{est-ext}
\int_{B_{r}(z_{0})} |\nabla \tilde{v}_{n}|^{2}\leq c_{n}r^{1+2\alpha_{n}},
\end{equation}
where $\tilde{v}_{n}$ stands for the extension.

For $r<r_{n}$, denote by $z'_{r}$ and $z''_{r}$ the two points on $\Gamma_{g_{n}}\cap \partial B_{r}(z_{0})$ such that
the
open sub-arcs of $\Gamma_{g_{n}}$ with end points $z'_{r}$, $z_{0}$ and $z''_{r}$, $z_{0}$ are contained in
$\Gamma_{g_{n}}\cap \partial B_{r}(z_{0})$. Setting $z'_{r}=g_{n}(\theta'_{r})\sigma(\theta'_{r})$ and
$z''_{r}=g_{n}(\theta''_{r})\sigma(\theta''_{r})$, denote by $l$ the line segment joining $z'_{r}$ and
$z''_{r}$ and define
\[
\tilde{g}_{n}(\theta) :=
\begin{cases}
g_{n}(\theta) & \theta \in [0, 2\pi)\setminus (\theta'_{r},\theta''_{r}) \\
\min\{ h(\theta)+\frac{1}{n} , l(\theta)\} & \theta \in (\theta'_{r},\theta''_{r}),
\end{cases}
\]
where $l(\theta)$ is the polar representation of $l$.

By \eqref{est-ext} and by the minimality of the pair $(g_{n},v_{n})$ we have
\begin{equation}\label{half-1-Tamanin}
\Ha^{1}(\Gamma_{g_{n}}\cap B_{r}(z_{0}))-\Ha^{1}(\Gamma_{\tilde{g}_{n}}\cap B_{r}(z_{0})) \leq C_{n}r^{1+2\alpha_{n}}.
\end{equation}
Indeed we can estimate
\[
\begin{split}
0  \geq & \, \F(g_{n},v_{n}) - \F(\tilde{g}_{n},\tilde{v}_{n}) +\Lambda\left(  \big| |\Omega_{g_{n}}|-|\Omega_{h}|\big|
-  \big| |\Omega_{\tilde{g}_{n}}|-|\Omega_{h}|\big|\right) \\
& + c \left( \sqrt{(\abs{\Omega_{g_{n}}\Delta \Omega_{h}} - \eps_n )^2+\eps_n }  - 
\sqrt{(\abs{\Omega_{\tilde{g}_{n}}\Delta \Omega_{h}} - \eps_n )^2+\eps_n }  \right)\\
 \geq & \, \Ha^{1}(\Gamma_{g_{n}}\cap B_{r}(z_{0}))-\Ha^{1}(\Gamma_{\tilde{g}_{n}}\cap B_{r}(z_{0}))  -
\int_{B_{r}(z_{0})}Q(E(\tilde{v}_{n}))dz - (\Lambda+1)\pi r^{2} \\
 \geq & \, \Ha^{1}(\Gamma_{g_{n}}\cap B_{r}(z_{0}))-\Ha^{1}(\Gamma_{\tilde{g}_{n}}\cap B_{r}(z_{0})) -
C_{n}r^{1+2\alpha_{n}}
\end{split}
\]

We will show later that
\begin{equation}
\label{half-2-Tamanin}
\Ha^{1}(\Gamma_{\tilde{g}_{n}}\cap B_{r}(z_{0}))-\Ha^{1}({l}) \leq C r^{2}.
\end{equation}
Now the inequality \eqref{half-2-Tamanin} together with \eqref{half-1-Tamanin} gives us
\begin{equation}\label{Tamanin}
\Ha^{1}(\Gamma_{g_{n}}\cap B_{r}(z_{0}))-\Ha^{1}({l}) \leq C r^{1+2\alpha_{n}}
\end{equation}
and the desired $C^{1,\alpha}$-regularity follows from a classical result for quasiminimizers of the area functional
(see 
Theorem 1 in \cite{tamanini}) once we observe that
\[
\Ha^{1}({l})= \inf \left\{ P(F ; B_{r}(z_{0})) \;:\; F\Delta \Omega_{g_{n}} \Subset  B_{r}(z_{0})  \right\}.
\]
The proof of \eqref{half-2-Tamanin} is a consequence of the $C^{2}$-regularity of $h$ and goes as follows (see Figure
\ref{fig:line}):
\begin{equation*}
\begin{split}
\Ha^{1}(\Gamma_{\tilde{g}_{n}}\cap B_{r}(z_{0}))-\Ha^{1}({l})  & \leq \int_{\theta'_{r}}^{\theta''_{r}}
\sqrt{(\tilde{g}_{n}(\theta))^{2}+(\tilde{g}_{n}'(\theta))^{2}} - \sqrt{(l(\theta))^{2}+(l'(\theta))^{2}}\,d\theta  \\
& \leq \frac{1}{\gamma}  \int_{\theta'_{r}}^{\theta''_{r}} ( \tilde{g}_{n}^{2} - l^{2} ) \,  d\theta + \frac{1}{\gamma} 
\int_{\theta'_{r}}^{\theta''_{r}} (\tilde{g}_{n}'+l')(\tilde{g}_{n}'-l')\,d\theta \\
& \leq \frac{1}{\gamma}  |B_{r}(z_{0})| +  \frac{C}{\gamma}  \int_{\theta'_{r}}^{\theta''_{r}}
\abs{\tilde{g}_{n}'-l'}\,d\theta,
\end{split}
\end{equation*}
where $C$ depends on the Lipschitz norm of $\tilde{g}_{n}$ and $l$ in the interval $(\theta'_{r},\theta''_{r})$ and
$\gamma$ is a positive constant with $\gamma<\min_\R h$.

\begin{figure}[ht]
  \caption{}\label{fig:line}
  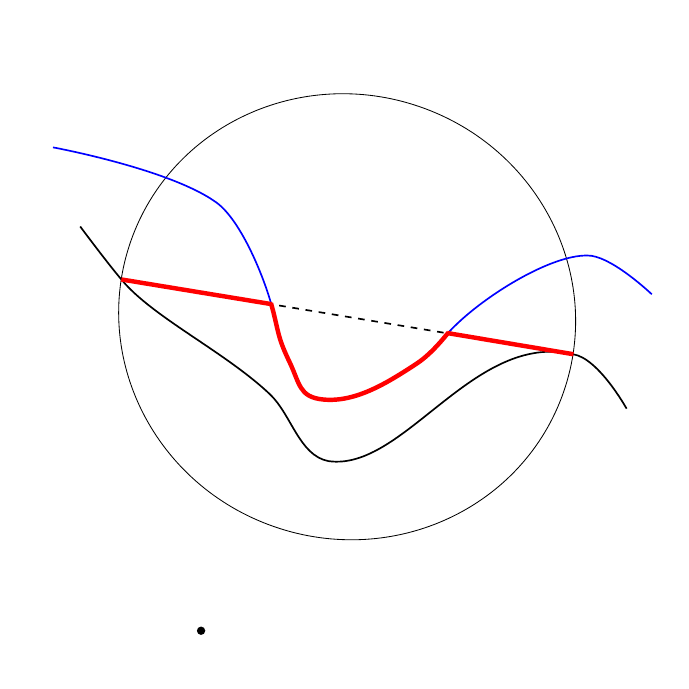
\end{figure}

To estimate the last term we first note that either the set $\{h+1/n < l\}$ is empty or there exists $\theta_{0} \in
(\theta'_{r},\theta''_{r})$ such that $\tilde{g}_{n}'(\theta_{0})-l'(\theta_{0})=0$ and using a second order Taylor
expansion around
$\theta_{0}$ we easily get
\[
 \int_{\theta'_{r}}^{\theta''_{r}} \abs{\tilde{g}_{n}'-l'}\,d\theta \leq C r^{2}
\]
where $C$ depends on the $C^{2}$-norm of $h$.

Now we claim that $g_{n}$ converges to $h$ in the $C^{1}$-norm. As in the proof of Lemma \ref{lem-balltohaus} we will
work
in the plane $(\theta,\rho)$ and we recall that the subgraph of $g_{n}$, still denoted by $F_{g_{n}}$, satisfies the
uniform exterior ball
condition. From the $C^1$-regularity and the uniform Lipschitz estimate, in the Lemma \ref{lem-balltohaus}, we obtain
$\sup_n ||g_n||_{C^1} < \infty $. Hence, from the uniform exterior ball condition we conclude that 
 at every point there exists a parabola touching $g_n$ from above. In other words, there is $C>0$ such that for every
$\theta_0$ it holds
for $P(\theta) = g_n(\theta_0) + g_n'(\theta_0)(\theta - \theta_0) + C\,  (\theta - \theta_0)^2$ that
\[
\min_{\theta} (P- g_n) = (P-g_n)(\theta_0) = 0.
\]  
This implies that the $g_n$ are uniformly semiconcave, i.e., for every $n$ the function
\[
\theta \mapsto g_n(\theta) - C \, \theta^2
\]
is concave. We may now use Lemma \ref{conc-conv} to conclude the desired $C^1$-convergence of $g_n$.  

The convergence of  $g_{n}$ to $h$ in $C^{1}$-norm allows us to use a blow-up method (see 
\cite{Fusco:2009ug}*{Theorem~ 6.10}) to infer the uniform estimate
\begin{equation}\label{est-unif}
\int_{B_{r}(z_{0})} |\nabla v_{n}|^{2}\leq c_{0}r^{1+2\sigma}
\end{equation}
for any $\sigma\in (1/2,1)$ and for all $r<r_{0}$ where $c_{0}$ and $r_{0}$ are independent of $n$.

Once we have \eqref{est-unif}, we can repeat the argument used to prove \eqref{Tamanin}, replacing
 \eqref{est-ext} by \eqref{est-unif}, to infer
\[
\Ha^{1}(\Gamma_{g_{n}}\cap B_{r}(z_{0}))-\Ha^{1}({l}) \leq C r^{1+2\sigma}.
\]
This implies a uniform estimate for the $C^{1,\alpha}$-norms of $g_{n}$ for $\alpha \in (0,1/2)$
(see for instance \cite{Cicalese:2010uc}*{Proposition 2.2}). The $C^{1,\alpha}$-convergence of $g_{n}$ now follows by a
compactness argument.

To conclude the proof we have just to observe that, since $v_{n}$ is a solution of the Lam\'e system in
$\Omega_{g_{n}}$, we can apply the elliptic estimates provided in \cite{Fusco:2009ug}*{Proposition 8.9} to deduce that
$\nabla v_{n}$ is uniformly bounded in $C^{0,\alpha}(\overline{\Omega}_{g_n},\R^{2}\times\R^{2})$ for all $\alpha \in
(0,1/2)$.
\end{proof}

\begin{lemma}\label{lem-convergence}
Let $(h,u) \in X_{\reg}(u_0)$ be a critical point of $\F$ such that $ 0< h < R_0$,
 and  let  $(g_{n},v_{n}) $ be as in the previous lemma with $\abs{\Omega_{g_n}\Delta \Omega_{h}}= o(\sqrt{\eps_n})$ if
$\eps_n$ is not identically zero and $\abs{\Omega_{g_n}\Delta \Omega_{h}}= o(1)$ if $\eps_n=0$ for all $n$. Suppose
that $\nabla
v_{n}\rightharpoonup \nabla u$ weakly in $L^{2}_{\loc}(\Omega_{h};\R^{2}\times \R^{2})$ and
\[
\lim_{n\to \infty} \int_{\Omega_{g_{n}}} Q(E(v_{n}))dz = \int_{\Omega_{h}} Q(E(u))dz.
\] 
 Then $g_{n}\in C^{1,1}(\R)$ and $g_{n}\to h$ in $C^{1,1}(\R)$,  for $n$ sufficiently large. 
\end{lemma}

\begin{proof}
From Lemma \ref{lem-more-regular} we know that $g_n \to h $ in $ C^{1, \alpha}(\R)$. Therefore for large $n$ there
exist diffeomorphisms  $\Phi_n: \bar{\Omega}_{g_n} \to \bar{\Omega}_h $ such that $\Phi_n \to id$ in $C^{1, \alpha}$.
Let $B_R$ be any ball of radius $R \in (R_0 - \max_\R h , R_0)$. Since 
\[
\sup_{n\in\N}\, \left\{{|| v_n ||}_{C^{1, \alpha}(\bar{\Omega}_{g_n})}\right\}< \infty
\]
by the convergence $\nabla v_n \rightharpoonup
\nabla u$ we have that 
\begin{equation}
\label{convergence.v_n}
\nabla v_n \circ \Phi_n^ {-1} \to \nabla u \qquad   \text{in} \,\, C^{0, \alpha}(\bar{\Omega}_h \cap B_R; \,  \Ma^{2
\times
2}).
\end{equation}

To prove the claim set $I_n := \{ \theta \in [0, 2\pi] \mid g_n (\theta) < h(\theta) + 1/n =: h_n(\theta)\}$. Since
$I_n$ is open, we may write $I_n = \bigcup_{i=1}^{\infty} (a_i^n, b_i^n)$. Notice that 
\begin{equation}
\label{outside.I_n}
g_n'(\theta)=  h_n'(\theta) = h'(\theta) \qquad \text{on} \,\,  [0, 2 \pi] \setminus I_n.
\end{equation} 
If $I_n$ is empty, the claim is trivial. Therefore we may assume that $I_n\neq\emptyset$.  Since $g_n
\in C^{1, \alpha}(\R)$, we can write the Euler-Lagrange equation for $(g_n,v_n)$ in the weak sense:
\begin{equation}
\label{penalized.eul}
k_{g_n}(\theta) = Q(E(v_n))(\theta, g_n(\theta))  + \beta_n(\theta, g_n(\theta)) + \lambda_n, \qquad \theta \in I_n.
\end{equation}
Here 
\[
\beta_n =  \frac{\Lambda \, \abs{\Omega_{g_n}\Delta \Omega_{h}}}{ \sqrt{(\abs{\Omega_{g_n}\Delta \Omega_{h}} -
\eps_n )^2+\eps_n }} \, \text{sign}\, (\chi_{\Omega_h} - \chi_{\Omega_{g_n}} )
\]
and $\lambda_n$ is some Lagrange multiplier. Notice that from the assumptions it follows that 
\begin{equation}
\label{beta-to-zero}
| \beta_n | = \frac{\Lambda \, \abs{\Omega_{g_n}\Delta \Omega_{h}}}{ \sqrt{(\abs{\Omega_{g_n}\Delta \Omega_{h}} -
\eps_n )^2+\eps_n }} \leq \Lambda \, \frac{\abs{\Omega_{g_n}\Delta \Omega_{h}}}{\sqrt{\eps_n}} \to 0.
\end{equation}

Recall the Euler-Lagrange equation for $(h, u)$
\begin{equation}
\label{euler-revisit}
k_{h}(\theta) = Q(E(u))(\theta, h(\theta)) + \lambda_{\infty}\,.
\end{equation}
We will show that $\lambda_n \to \lambda_{\infty}$. Notice that for the curvature in polar coordinates it holds that
\[
k_{g_n} \,  g_n = \frac{g_n^2+ 2 g_n'^2- g_n g_n''}{(g_n^2 + g_n'^2)^{\frac{3}{2}}} \, g_n =  -   \left(
\frac{g_n'}{\sqrt{g_n^2 +g_n'^2}} \right)' + \frac{g_n}{\sqrt{g_n^2 +g_n'^2}}.
\]
Hence, multiplying \eqref{penalized.eul} by $g_n$, integrating over $I_n$ and using \eqref{euler-revisit} yield 
\[
\begin{split}
  \int_{I_n}  & \Bigl[ Q(E(v_n)) \bigl(\theta, g_n(\theta)\bigr)  + \beta_n\bigl(\theta, g_n(\theta)\bigr) + \lambda_n
\Bigr] \,g_n \, d
\theta = \int_{I_n}  k_{g_n} \,g_n \,  d \theta \\
&= \int_{I_n}  -   \left( \frac{g_n'}{\sqrt{g_n^2 +g_n'^2}} \right)' + \frac{g_n}{\sqrt{g_n^2 +g_n'^2}}\,  d \theta \\
&= \sum_{i=1}^{\infty}   -   \left( \frac{g_n'(b_i^n)}{\sqrt{g_n^2(b_i^n) +g_n'^2(b_i^n)}}- 
\frac{g_n'(a_i^n)}{\sqrt{g_n^2(a_i^n)
+g_n'^2(a_i^n)}} \right) +  \int_{a_i^n}^{b_i^n}  \frac{g_n}{\sqrt{g_n^2 +g_n'^2}}\,  d \theta \\
&= \sum_{i=1}^{\infty}    -   \left( \frac{h_n'(b_i^n)}{\sqrt{h_n^2(b_i^n) +h_n'^2(b_i^n)}}- 
\frac{h_n'(a_i^n)}{\sqrt{h_n^2(a_i^n)
+h_n'^2(a_i^n)}} \right) +  \int_{a_i^n}^{b_i^n}  \frac{g_n}{\sqrt{g_n^2 +g_n'^2}}\,  d \theta \\
&= \int_{I_n}  k_{h_n} \,h_n \,  d \theta +  \int_{I_n}  \frac{g_n}{\sqrt{g_n^2 +g_n'^2}} - \frac{h_n}{\sqrt{h_n^2
+h_n'^2}}\,\,  d \theta  \\
&= \int_{I_n} \Bigl[ Q(E(u))\bigl(\theta, h(\theta)\bigr)  + \lambda_\infty \Bigr] \,h \,  d \theta +  \int_{I_n}
(k_{h_n} \,h_n -
k_{h}
\,h) + \frac{g_n}{\sqrt{g_n^2 +g_n'^2}} - \frac{h_n}{\sqrt{h_n^2 +h_n'^2}}\,\,  d \theta  .
\end{split}
\]
Recall that $h_n = h + 1/n$. Therefore by \eqref{convergence.v_n}, \eqref{beta-to-zero} and the previous
calculations  we conclude that 
\[
\lim_{n \to \infty}\frac{1}{|I_n|} \int_{I_n} \lambda_n g_n - \lambda_\infty h \, d \theta = 0\,,
\]
which clearly implies $\lambda_n \to \lambda_\infty$. 

From \eqref{outside.I_n} and \eqref{penalized.eul} we conclude that $g_n \in C^{1,1}(\R)$. Moreover by the 
equations  \eqref{outside.I_n}, \eqref{penalized.eul} and  \eqref{euler-revisit} together with the convergences 
\eqref{convergence.v_n}, \eqref{beta-to-zero} and $\lambda_n \to \lambda_\infty$ we conclude that
\[
 k_{g_n}  \to k_{h} \qquad \text{in } L^{\infty}.
\] 
This in turn gives us the convergence
\[
g_n'' \to h'' \qquad \text{in }  L^{ \infty}.
\]
\end{proof}

Now we are in  position to prove the main theorem of this section.
\begin{proof}[Proof of Theorem \ref{TH:criterion}]
\textbf{Step 1:} We show first that $(h,u)$ is a strict local minimizer, i.e., we prove the claim without the estimate on the
right-hand side of \eqref{main-equation}. 

Observe that from the results  of the previous section we may assume that $(h,u)$ is a
$C^{1,1}$-local minimizer. The result will follow once we prove that the $C^{1,1}$-local minimality implies
the local minimality. Arguing by contradiction let us assume that for any $n\in \N$ there exist $(h_{n},u_{n})\in 
X(u_{0})$ with $|\Omega_{h_{n}}|=|\Omega_{h}|$ such that 
\[
\F(h_{n},u_{n}) \leq \F(h,u) \quad \text{ and }\quad 0< d_{\Ha}(\Gamma_{h_{n}}\cup\Sigma_{h_n},\Gamma_{h})\leq \frac{1}{n}.
\]
Consider the sequence $(g_{n},v_{n})\in X(u_{0})$ of minimizers of the following penalized obstacle problem
\[
\min\left\{  \F(g,v)+\Lambda \big| |\Omega_{g}|-|\Omega_{h}|\big| \;:\; (g,v)\in X(u_{0}),\, g\leq h+\frac{1}{n}
\right\}\,,
\]
for  some large $\Lambda $. Since $(h_n,u_n)$ and $(h,u)$ are clearly competitors, we have that
\[\F(g_{n},v_{n})\leq \F(h_n,u_n) \leq\F(h,u)\,.\]
By the contradiction assumption we may assume that $(h_n,u_n) \neq (h,u)$.

By the compactness property of $X(u_{0})$ there exists $(g,v)$ such that, up to subsequences,
$(g_{n},v_{n})\to (g,v)$ in $X(u_{0})$. Let $(f,w)\in X(u_{0})$ with $f\leq h$, by the lower semicontinuity of $\F$ and
the minimality of $(g_{n},v_{n})$, we get
\begin{equation}
\label{minimality}
\begin{split}
\F(g,v)+\Lambda\big| |\Omega_{g}|-|\Omega_{h}| \big|  & \leq \liminf_{n\to \infty} \Big[  \F(g_{n},v_{n}) +\Lambda\big| |\Omega_{g_{n}}|-|\Omega_{h}| \big|\Big]  \\ 
& \leq \F(f,w)+\Lambda\big| |\Omega_{f}|-|\Omega_{h}| \big|.
\end{split}
\end{equation}
Choosing $(f,w)=(h,v)$ in the previous inequality, we obtain that
\begin{equation}\label{controlemma}
\Ha^{1}(\Gamma_{g})+\Lambda \big| |\Omega_{g}|-|\Omega_{h}| \big| \leq \Ha^{1}(\Gamma_{h})
\end{equation}
When $\Lambda$ is sufficiently large, \eqref{controlemma} and Lemma \ref{lemma-isoperimetric} imply that $g=h$. Moreover, we observe that from \eqref{minimality} it follows that $(h,v)$ minimizes $\F$ in the class of all $(f,w)\in X(u_{0})$ with $f= h$. In particular $v$ must coincide with the elastic equilibrium
$u$.

Choosing $(f,w)=(h,u)$ in \eqref{minimality}, using the lower
semicontinuity of  $g\mapsto \mathcal H^1(\Gamma_g)$ with respect to the $L^1$-convergence and the lower semicontinuity of 
the elastic energy with respect to the weak $H^1$-convergence, we deduce 
\[
\lim_{n\to \infty}\Ha^{1}(\Gamma_{g_{n}}\cup \Sigma_{g_{n}})=\Ha^{1}(\Gamma_{h}),
\]
\[
\lim_{n\to \infty}\int_{\Omega_{g_{n}}} Q(E(v_{n}))\,dz=\int_{\Omega_{h}} Q(E(u))\,dz.
\]
From Lemma \ref{lem-convergence} we get $g_{n}\to h$ in $C^{1,1}(\R)$. 

We only need to modify $g_{n}$ such that it satisfies the volume constraint. We simply define
$\tilde{g}_n(\theta):= g_n(\theta) + \delta_n$ where $ \delta_n $ are chosen so that $|\Omega_{\tilde{g}_n}|
=|\Omega_h|$. By Lemma \ref{volume-lemma} it holds $|\Omega_{g_n}| \geq |\Omega_h|$ and therefore $\delta_n\geq 0$ and
$\Omega_{\tilde{g}_n} \subset \Omega_{g_n}$. Hence $v_n$ is well defined in $\Omega_{\tilde{g}_n}$ and $(\tilde{g}_n,
v_n)$ is an admissible pair. 

Since $h > 0$ and $g_n \to h$ uniformly, we have $g_n >\gamma$ for some $\gamma>0$ independent from $n$ and $\delta_n
\to 0$. We may estimate
\[
\begin{split}
\Ha^{1}(\Gamma_{\tilde{g}_n})-\Ha^{1}(\Gamma_{g_n}) &= \int_0^ {2 \pi} \sqrt{(g_n + \delta_n)^ 2 + g_n'^ 2} -   \sqrt{g_n^ 2 + g_n'^ 2} \, d \theta \\
&\leq \frac{1}{\gamma}  \int_0^ {2 \pi}  2g_n \delta_n + \delta_n^ 2  \, d \theta
\end{split}
\]
and
\[
 \big| |\Omega_{\tilde{g}_n}| - | \Omega_{g_n}| \big| = \frac{1}{2} \int_0^ {2 \pi} (g_n + \delta_n)^ 2  -   g_n^ 2 \, d \theta =   \frac{1}{2} \int_0^ {2 \pi} 2g_n \delta_n + \delta_n^ 2  \, d \theta.
\]
Therefore whenever $\Lambda \geq \frac{2}{\gamma}$ we have
\begin{equation}
\label{absorbed}
 \Ha^{1}(\Gamma_{\tilde{g}_n})-\Ha^{1}(\Gamma_{g_n}) \leq \Lambda \big| |\Omega_{\tilde{g}_n}| - | \Omega_{g_n}| \big|.
\end{equation}

The claim now follows, since by the choice of $\tilde{g}_n$ and by \eqref{absorbed} we have
\[
\begin{split}
\F(\tilde g_{n}, v_{n})  &= \F(\tilde g_{n}, v_{n}) +  \Lambda \big| |\Omega_{\tilde g_{n}}| - | \Omega_h| \big| \\
&\leq  \F(g_{n},v_{n})  + \Lambda \big| |\Omega_{ g_n}| - | \Omega_h| \big| \leq \F(h_n,u_n) \leq \F(h,u)\,.
\end{split}
\]
This contradicts the fact that $(h,u)$ is a strict  $C^{1,1}$-local minimizer.

\noindent\textbf{Step 2:} We will now prove the theorem. 
The proof is very similar to the first step. Arguing by contradiction we assume that there are  $(h_{n},u_{n})\in 
X(u_{0})$ with $|\Omega_{h_{n}}|=|\Omega_{h}|$ such that 
\[
\F(h_{n},u_{n}) \leq \F(h,u) + c_0 \,|\Omega_{h_n} \Delta \Omega_h|^2
\quad \text{ and }\quad 0 < d_{\Ha}({\Gamma}_{h_{n}}\cup\Sigma_{g_n},\Gamma_{h})\leq \frac{1}{n}\,.
\]
Denote $\eps_n := |\Omega_{h_n} \Delta \Omega_h|$. Notice that $d_{\Ha}(\Gamma_{h_{n}}\cup\Sigma_{g_n},\Gamma_{h})
\to 0$ implies $\chi_{\Omega_{h_n}} \to \chi_{\Omega_{h}}$ in $L^1$ and therefore $\eps_n \to 0$. 

This time we replace the contradicting sequence $(h_{n},u_{n})$ by  $(g_{n},v_{n})\in X(u_{0})$ which  minimizes 
\[
\min\left\{  \F(g,v)+\Lambda \big| |\Omega_{g}|-|\Omega_{h}|\big| +  \sqrt{(\abs{\Omega_{g}\Delta \Omega_{h}} -
\eps_n )^2+\eps_n } \;:\; (g,v)\in X(u_{0}),\,g\leq h+\frac{1}{n} \right\}\,.
\]
By compactness we may assume that, up to a subsequence, $(g_{n},v_{n})\to (g,v)$ in $X(u_{0})$. By a completely similar 
argument as in Step 1 we conclude that $(g,v) = (h,u)$
whenever $\Lambda$ is sufficiently large. Moreover, we have that
\[
\lim_{n\to \infty}\Ha^{1}(\Gamma_{g_{n}}\cup \Sigma_{g_{n}})=\Ha^{1}(\Gamma_{h}),
\]
\[
\lim_{n\to \infty}\int_{\Omega_{g_{n}}} Q(E(v_{n}))\,dz=\int_{\Omega_{h}} Q(E(u))\,dz.
\]
To conclude that $g_{n}\to h$ in $C^{1,1}(\R)$,  we will prove that
 \begin{equation}
\label{conv-to-1}
\lim_{n \to \infty} \frac{\abs{\Omega_{g_n}\Delta \Omega_{h}}}{\eps_n} =1
\end{equation}
and apply Lemma \ref{lem-convergence}.

Suppose that \eqref{conv-to-1} were false. Then there exists $c > 0$ such that $\big| \abs{\Omega_{g_n}\Delta
\Omega_{h}} - \eps_n \big|  \geq c \, \eps_n $.  Using the minimality of $(g_n,v_n)$ and the contradiction assumption
for $(h_n,u_n)$, we obtain
\begin{equation}\label{long.inequality}
\begin{split}
 \F(g_n,v_n)+ &\Lambda \big| |\Omega_{g_n}|-|\Omega_{h}|\big| +  \sqrt{(\abs{\Omega_{g_n}\Delta \Omega_{h}} - \eps_n )^2+\eps_n } \\
 &\leq  \F(h_n,u_n)+   \sqrt{\eps_n } \\
&<  \F(h,u)+ c_0 \eps_n^2 +  \sqrt{\eps_n }.
\end{split}
\end{equation}

Now we observe that from \cite{FFLmill}*{Proposition 6.1}, for $\Lambda$ sufficiently large, $(h,u)$ is also a minimizer of the penalized problem
\[
\F(g,v)+ \Lambda \big| |\Omega_{g}|-|\Omega_{h}|\big|.
\]
Hence we have
\begin{equation}\label{long.inequality-2}  
\F(h,u) \leq  \F(g_n,v_n)+ \Lambda \big| |\Omega_{g_n}|-|\Omega_{h}|\big| .
\end{equation}
Combining \eqref{long.inequality} and \eqref{long.inequality-2} we get
\[
\sqrt{ c^2 \eps_n^2 + \eps_n} \leq    \sqrt{(\abs{\Omega_{g_n}\Delta \Omega_{h}} - \eps_n )^2+\eps_n } <c_0\,  \eps_n^2 +  \sqrt{\eps_n },
\]
which is a contradiction since $\eps_n \to 0$ proving \eqref{conv-to-1}. 

Arguing  as in \eqref{long.inequality} and by using \eqref{conv-to-1} we obtain 
\begin{equation}
\label{almost.there}
\begin{split}
\F(g_n,v_n)+ \Lambda \big| |\Omega_{g_n}|-|\Omega_{h}|\big| &\leq    \F(h_n,u_n)+  \sqrt{\eps_n } - \sqrt{(\abs{\Omega_{g_n}\Delta \Omega_{h}} - \eps_n )^2+\eps_n } \\
&<  \F(h,u)+ c_0 \eps_n^2 \\
&\leq   \F(h,u)+ 2 c_0 \, \abs{\Omega_{g_n}\Delta \Omega_{h}}^2 ,
\end{split}
\end{equation}
when $n$ is large.

As in Step 1 define
$\tilde{g}_n(\theta):= g_n(\theta) + \delta_n$ where $ \delta_n \geq 0$ are such that $|\Omega_{\tilde{g}_n}| =|\Omega_h|$. By choosing $\Lambda$ large enough we have
\begin{equation}
\label{absorbed2}
 \Ha^{1}(\Gamma_{\tilde{g}_n})-\Ha^{1}(\Gamma_{g_n}) \leq \frac{\Lambda}{2} \,  \big| |\Omega_{\tilde{g}_n}| - | \Omega_{g_n}| \big|.
\end{equation}
Therefore since 
\[
\abs{\Omega_{g_n}\Delta \Omega_{h}}^2 \leq  2 \abs{\Omega_{\tilde{g}_n}\Delta \Omega_{h}}^2  + 2 \abs{\Omega_{\tilde{g}_n}\Delta \Omega_{g_n}}^2 = 2 \abs{\Omega_{\tilde{g}_n}\Delta \Omega_{h}}^2  + 2 \big|  | \Omega_{g_n}| - |\Omega_{h}| \big|^2  
\]
we have by \eqref{almost.there} and  \eqref{absorbed2} that
\[
\begin{split}
\F(\tilde{g}_n,v_n) &\leq  \F(g_n,v_n)+ \frac{\Lambda}{2} \big| |\Omega_{g_n}|-|\Omega_{h}|\big| \\
& < \F(h,u)+ 2 c_0 \, \abs{\Omega_{g_n}\Delta \Omega_{h}}^2 - \frac{\Lambda}{2} \big|
|\Omega_{g_n}|-|\Omega_{h}|\big| \\
&\leq  \F(h,u)+ 4 c_0 \, \abs{\Omega_{\tilde{g}_n}\Delta \Omega_{h}}^2 - \frac{\Lambda}{2} \big| |\Omega_{g_n}|-|\Omega_{h}|\big| + 4 c_0 \,   \big|  | \Omega_{g_n}| - |\Omega_{h}| \big|^2 \\
&\leq \F(h,u)+ 4 c_0 \, \abs{\Omega_{\tilde{g}_n}\Delta \Omega_{h}}^2 ,
\end{split}
\]
when $n$ is sufficiently large. This contradicts Proposition \ref{H2-localmin} when $c_0$ is chosen to be small enough.
\end{proof}

\section{The case of the disk}\label{Sec:6}

In this section we consider the particular case when a radial stretching is applied to a material with round cavity
$F=\bar{B}_r$. We prove that the disk remains stable under small radial stretching. This result is similar to the case
of flat configuration in \cite{Fusco:2009ug}. The main difference to the flat case, where the minimal shape is a
rectangle, is that the curvature of the disk is nonzero and therefore the second variation formula becomes considerably
more complicated. Instead of trying to explicitly write the second variation, we use fine estimates to find a range of
stability.

The Dirichlet boundary condition has the form of radial stretching,
\begin{equation}
\label{dirichlet}
u_0\bigl(\rho\sigma(\theta)\bigr) = \alpha R_0 \sigma(\theta)\qquad \text{ for }\rho\geq R_0\,,
\end{equation}
where $\alpha \in \R$ is some constant. The region occupied by the elastic
material is the annulus $A(R_0,r) := B_{R_0} \setminus \bar{B}_r$. For $u_0$ as above 
we say that $(h,u)\in X(u_0)$ is a \emph{round configuration} if $h(\theta)\equiv r$ and $u$ is the elastic equilibrium
associated to $h$.

For the next theorem we define
\[
 \beta(t):=1+\frac{\mu+\lambda}{\mu}\frac{t^2}{R_0^2}\,.
\]
Recall also the definition of the ellipticity constant $\eta=\min \{\mu,\mu+\lambda\}$.

\begin{theorem}
\label{RoundIsMin}
Let 
\begin{equation*}
\label{def.r_1}
r_0 : = \sup \left\{ t \leq R_0 \mid (1+ t^2) \log\left( \frac{R_0}{t}  \right) \geq \frac{\eta}{4 \mu}\right\}
\end{equation*}
and define the function $G:\R\to [-\infty,R_0)$ as
\begin{equation*}
\label{def.r_2}
G(\alpha): = \sup\left\{  t \leq R_0 \mid  t \log\left(\frac{R_0}{t} \right)\beta^2(t) \geq \frac{\eta}{32  (\mu +
\lambda)^2 \alpha^2} \right\}\,.
\end{equation*}
If $r\in(r_0,R_0)$ and $\alpha\in\R$ satisfy 
\begin{equation}\label{vesacond}
 r>G(\alpha)\,,
\end{equation}
then the round configuration is a strict local minimizer of $\F$ under the volume constraint.
\end{theorem}

The elastic equilibrium $u$ can be explicitly calculated. Indeed, because of the symmetry we can write
\[u(\rho\sigma(\theta)) = f(\rho)\sigma(\theta) 
\]
and applying the first equation in \eqref{euler} we have
\[
f''(\rho) + \frac{f'(\rho)}{\rho} - \frac{f(\rho)}{\rho^2} = 0\,.
\]
This can be easily solved 
\[
f(\rho) = \frac{a}{\rho} + b \rho,
\]
for some $a,b \in \R$. To find $a$ and $b$ observe that
\begin{equation}
\label{GenCEu}
\begin{split}
\C E(u) =  2 \mu \begin{pmatrix}   f'(\rho) \cos^2 \theta + \frac{f(\rho)}{\rho} \sin^2 \theta     &     \left( f'(\rho)  - \frac{f(\rho)}{\rho} \right) \sin \theta \cos \theta \\
			   \left( f'(\rho)  - \frac{f(\rho)}{\rho} \right) \sin \theta \cos \theta &   f'(\rho) \sin^2 \theta + \frac{f(\rho)}{\rho} \cos^2 \theta 
\end{pmatrix} + \lambda \left( f'(\rho) + \frac{f(\rho)}{\rho} \right) \begin{pmatrix} 1 & 0 \\
0 & 1 \end{pmatrix}.
\end{split}
\end{equation}
Therefore, the second equation in \eqref{euler} gives  
\[
(2 \mu + \lambda ) f'(r) + \lambda \frac{f(r)}{r} = 0.
\]
This and the Dirichlet condition \eqref{dirichlet} yield
\begin{equation}
\label{bee}
 \frac{a}{r^2} =\frac{\mu + \lambda }{\mu} \,  b \quad \text{and} \quad b= \frac{\alpha}{\beta(r)}.
\end{equation}

It is trivial to check that the round configuration is a critical point of $\F$. To prove Theorem \ref{RoundIsMin}
we need to show that the round configuration is a point of positive second variation. To this aim, let us
explicitly write the
quadratic form \eqref{biliear}.  By \eqref{GenCEu} and
\eqref{bee} we have 
\[
\C E(u) =  4b( \mu + \lambda)  \begin{pmatrix}   \sin^2 \theta     &   - \sin \theta \cos \theta \\
							- \sin \theta \cos \theta  & \cos^2 \theta 
\end{pmatrix} = 4b (\mu + \lambda ) \, \, \tau \otimes \tau,
\]
on the boundary $\partial B_r$. Hence, for $\psi\in H^1_\sharp (\R)$, we have
\[
\diver_{\tau} ( \langle \vect{\psi}, \nu \rangle\C E(u)) = 4b(\mu + \lambda ) (-  \langle \vect{\psi}, \nu \rangle \nu +  \partial_{\tau} \langle \vect{\psi}, \nu \rangle \tau)
\]
and the equation \eqref{u_psi} for $u_{\psi}$ becomes
\begin{equation}
\label{udotflat}
\int_{A(R,r)} \C E(u_{\psi}) : E(w) \, dz = - 4b(\mu + \lambda ) \int_{\partial B_r}  \left(-  \langle \vect{\psi}, \nu \rangle \langle w, \nu \rangle +  ( \partial_{\tau} \langle \vect{\psi}, \nu \rangle ) \langle w, \tau \rangle \right) \, d \Ha^1.
\end{equation}
Moreover, in the case of a round configuration the elastic energy is
\begin{equation}
\label{energycal}
Q(E(u)) = 2( \mu + \lambda ) b^2+ 2 \mu \frac{a^2}{\rho^4} 
\end{equation}
and therefore, by \eqref{bee}, we get
\[
\partial_{\nu} Q(E(u)) = - \frac{8(\mu + \lambda)^2}{\mu} \frac{b^2}{r} \qquad \text{on} \quad \partial B_r.
\]
Hence, \eqref{biliear} becomes
\begin{equation}
\label{disk2}
\begin{split}
\partial^2 \F(h, u)[\psi] =  &- \int_{A(R,r)} 2  Q(E(u_{\psi}))  \,  dz + \int_{ \partial B_r} |  \partial_{\tau}
\langle \vect{\psi}, \nu \rangle |^2 \,  d \Ha^1\\
&+ \int_{ \partial B_r} \left(  \frac{8(\mu + \lambda)^2}{\mu} \frac{b^2}{r}  -   \frac{1}{r^2} \right)\,  \langle \vect{\psi}, \nu \rangle^2 \,  d \Ha^1, 
\end{split}
\end{equation}
where $u_{\psi} \in \A(B_R \setminus \bar{B}_r)$ solves (\ref{udotflat}), and $\psi$ satisfies
$\int_0^{2 \pi} \psi \, d \theta = 0$.

Now the goal is to prove that  \eqref{disk2} is positive whenever the assumptions of Theorem \ref{RoundIsMin} are
satisfied. The main obstacle is to bound the first term in \eqref{disk2} which will be done by using the equation
\eqref{udotflat}. To this aim we need the following simple lemma, which we prove to keep track of the optimal
constant. 
\begin{lemma}
\label{bdrest}
Suppose that $v \in  W^{1,2}(A(R_0,r); \R^2)$ is a continuous map with $v = 0$ on $\partial B_{R_0}$ and $A$ is a
matrix. Then for $w(z) = v(z) + Az$ we have that
\[
\int_{\partial B_r} | w|^2 \, d \Ha^1 \leq r \log \left(\frac{R_0}{r} \right) \, \int_{A(R_0,r)} 
\Big|D v -\frac{r}{R_0-r}A\Big|^2 \, dz.
\]

\end{lemma}

\begin{proof}
Consider $w$ in polar coordinates. Fix an angle $\theta$ and integrate over $[r,R_0]$ 
\[
  A  R_0\sigma(\theta) - w(r\sigma(\theta))  =  \int_r^{R_0}   Dw\left(\rho \sigma(\theta)\right) \sigma(\theta) \, d
\rho,
\]
which implies
\[
|w(r\sigma(\theta))| \leq   \int_r^{R_0} \Big| Dv(\sigma(\theta)) - \frac{r}{R_0-r}A\Big| \, d \rho.
\]
Integrate over $\theta$ and use H\"older's inequality to obtain
\[
\begin{split}
\int_0^{2 \pi} |w(\rho\sigma(\theta))|^2 \, d \theta &\leq \int_0^{2 \pi} \left( \int_r^{R_0}
\Big|Dv(\rho\sigma(\theta)) - \frac{r}{R_0-r}A\Big| \, d \rho \right)^2 \, d \theta \\
&\leq \int_0^{2 \pi} \left( \int_r^{R_0} \frac{1}{\rho}  \, d \rho \cdot \int_r^{R_0} \Big|Dv(\rho\sigma(\theta)) -
\frac{r}{R_0-r}A\Big|^2 \, \rho\, d \rho \right) \, d \theta \\
&=  \log \left(\frac{R_0}{r} \right) \, \int_{A(R_0,r)} \Big|Dv - \frac{r}{R_0-r}A\Big|^2 \, dz .
\end{split}
\]
The inequality follows from $\int_{\partial B_r} |w|^2 \, d \Ha^1 = r \int_0^{2 \pi} |w(r, \theta)|^2 \, d \theta$.
\end{proof}

\begin{proof}[Proof of Theorem \ref{RoundIsMin}]
As we stated before, by the local minimality criterion it is enough to prove that the second variation of $\F$ at
$(h,u)$ is positive. Suppose that  $\psi \in H_{\sharp}^1(\R)$ satisfies $\int_0^{2 \pi} \psi \, d \theta = 0$ and
$\psi \neq 0$.  Without loss of generality we may assume $\psi$ to be smooth. To estimate the first term in
\eqref{disk2} we claim that
\begin{equation}
\label{est.udot}
2 \int_{A(R_0,r)} Q( E(u_{\psi})) \, dz \leq \frac{32 ( \mu + \lambda)^2 b^2}{\eta} r \log\left( \frac{R_0}{r} \right)
\int_{\partial B_r}   \langle \vect{\psi}, \nu \rangle^2 +  | \partial_{\tau} \langle \vect{\psi}, \nu \rangle |^2 \, d
\Ha^1\,.
\end{equation}
To this aim, choose $w(z)= u_{\psi}(z) + Az$ as a
test function in \eqref{udotflat} where $A$ is antisymmetric, to obtain
\begin{equation}
\label{plug.test}
\begin{split}
2 \int_{A(R_0,r)} & Q( E(u_{\psi})) \, dz =- 4b(\mu + \lambda ) \int_{\partial B_r} \left(-  \langle \vect{\psi}, \nu
\rangle \langle w, \nu \rangle +   \partial_{\tau} \langle \vect{\psi}, \nu \rangle   \langle w, \tau \rangle \right) \,
d \Ha^1 \\
&\leq 4b(\mu + \lambda ) \left( \int_{\partial B_r}     \langle \vect{\psi}, \nu \rangle^2 +  | \partial_{\tau} \langle \vect{\psi}, \nu \rangle |^2  \, d \Ha^1 \right)^{1/2} \left(  \int_{\partial B_r}  | w |^2 \, d \Ha^1 \right)^{1/2}.
\end{split}
\end{equation}
Apply Lemma \ref{bdrest} to $w$ to get
\begin{equation}
\label{apply.lemma}
\int_{\partial B_r}  | w|^2 \, d \Ha^1 \leq  r \log \left(\frac{R_0}{r} \right) \, \int_{A(R_0,r)}
\Big|D u_{\psi} - \frac{r}{R-r}A\Big|^2 \, dz.
\end{equation}
Let $R_k \to \infty$ and for every $k$ choose an antisymmetric $A_k$ such that 
\[
\int_{A(R_k,r)} Du_{\psi}  - \frac{r}{R-r}A_k \, dz = \int_{A(R_k,r)} Du_{\psi}^T + \frac{r}{R-r}A_k \, dz. 
\]
By Theorem \ref{kornann} we get
\[ 
\int_{A(R_k,r)}  \Big|D u_{\psi} - \frac{r}{R-r}A_k\Big|^2 \, dz \leq C_k \int_{A(R_k,r)}  |E( u_{\psi})|^2 \, dz =  C_k
\int_{A(R,r)}  |E( u_{\psi})|^2 \, dz\,.
\]
Together with \eqref{apply.lemma} this yields
\[
\int_{\partial B_r}  | w|^2 \, d \Ha^1 \leq  r \log \left(\frac{R_0}{r} \right) C_k \int_{A(R_0,r)}  |E( u_{\psi})|^2 \,
dz.
\]
Since  $C_ k \to 4$ as $R_k \to \infty$ we have that
\[
\int_{\partial B_r}  |w|^2 \, d \Ha^1 \leq \frac{4r}{\eta}\log \left(\frac{R_0}{r} \right) \int_{A(R_0,r)}  Q(E(
u_{\psi})) \, dz.
\]
Now \eqref{est.udot} follows from \eqref{plug.test} and from the previous inequality.

We estimate \eqref{disk2} by using \eqref{est.udot} and obtain
\begin{equation}
\label{2nd_1st.est}
\begin{split}
\partial^2 \F(h, u)[\psi] &\geq  -  32 \eta^{-1} ( \mu + \lambda)^2 b^2 \, r \log\left( \frac{R_0}{r} \right)
\int_{\partial B_r} \langle \vect{\psi}, \nu \rangle^2 +  | \partial_{\tau} \langle \vect{\psi}, \nu \rangle |^2  \, d
\Ha^1 \\
&\, \quad+ \int_{ \partial B_r} |  \partial_{\tau} \langle \vect{\psi}, \nu \rangle |^2  \,  d \Ha^1  + \int_{ \partial B_r}  \left(  \frac{8(\mu + \lambda)^2}{\mu} \frac{b^2}{r}  -   \frac{1}{r^2} \right)\,  \langle \vect{\psi}, \nu \rangle^2 \,  d \Ha^1 \\
&=    \int_{ \partial B_r} |  \partial_{\tau} \langle \vect{\psi}, \nu \rangle  |^2  - \frac{1}{r^2}  \,  \langle \vect{\psi}, \nu \rangle^2 \,  d \Ha^1 \\
&\,\quad-  32 \eta^{-1} ( \mu + \lambda)^2 b^2 \, r \log\left( \frac{R_0}{r} \right) \int_{\partial B_r} 
|\partial_{\tau} \langle \vect{\psi}, \nu \rangle |^2  \, d \Ha^1 \\
&\quad+ \left( \frac{r}{\mu} - 4\eta^{-1} r^3 \log\left( \frac{R_0}{r} \right)\right) \,8 (\mu + \lambda)^2 b^2 \, 
\int_{ \partial B_r}  \frac{1}{r^2}  \langle \vect{\psi}, \nu \rangle^2\,  d \Ha^1 .
\end{split}
\end{equation}
Let us first treat the last term in (\ref{2nd_1st.est}). For
every $r > r_0$ we have that 
\[
\begin{split}
\partial^2 \F(h, u)[\psi] &>  \left(1- 32 \eta^{-1} ( \mu + \lambda)^2 b^2 \, r \log\left( \frac{R_0}{r} \right)
\right) \int_{
\partial B_r}  | \partial_{\tau} \langle \vect{\psi}, \nu \rangle |^2  - \frac{1}{r^2}  \langle \vect{\psi}, \nu
\rangle^2 \,  d \Ha^1 .
\end{split}
\]
Furthermore, if \eqref{vesacond} is satisfied, then
\[
1- 32 \eta^{-1} ( \mu + \lambda)^2 b^2 \, r \log\left( \frac{R_0}{r} \right) > 0\,.
\]
By the definition
\eqref{not.vect} we see that $\langle \vect{\psi}, \nu \rangle = \psi \left( \sigma^{-1} \left(\frac{z}{|z|} \right) 
\right)$.
Hence, by the Wirtinger's inequality, we get 
\[ 
\int_{ \partial B_r}  | \partial_{\tau} \langle \vect{\psi}, \nu \rangle |^2  - \frac{1}{r^2}  \langle \vect{\psi}, \nu
\rangle^2 \,  d \Ha^1 =  \frac{1}{r}\int_0^{2 \pi} | \psi'(\theta) |^2  -  |\psi(\theta)|^2\,  d \theta \geq 0\,.
\] 
which concludes the proof.
\end{proof}

At the end of the section we study the global minimality of the round configuration. We begin with the following remark.

\begin{remark}
\label{global}
Suppose that $R_0$ and $r_0$ are as in Theorem \ref{RoundIsMin} and fix  $\alpha \in \R$ and a small $\eps >0$. Then for every $r \in [r_0 + \eps, R_0 ]$ such that $r \geq G(\alpha) + \eps$ the proof above actually gives 
\[
\partial^2 \F(h, u)[\psi] \geq  c_2 \int_0^{2 \pi} | \psi'(\theta) |^2  -  |\psi(\theta)|^2\,  d \theta + c_1 \int_0^{2 \pi}  |\psi(\theta)|^2\,  d \theta,
\]   
for some small  $0< c_1 < c_2$, independent of $r$.  Using the Wirtinger's inequality we get
\[
\partial^2 \F(h, u)[\psi] \geq c_0 || \psi ||_{H^1([0,2\pi))}^2,
\]   
for $c_0$ depending only on  $R_0, r_0, \alpha$ and $\eps$.  This is a uniform version of Lemma \ref{easylemma}.

We can use this uniform bound of the constant $c_0$  to prove a uniform  local $C^{1,1}$-minimality of the round
configuration for  $r \in [r_0 + \eps, R_0]$ with $r \geq G(\alpha) + \eps$. Indeed, arguing as in Proposition
\ref{H2-localmin} and in Lemma \ref{2ndpostive} we conclude that there is $\delta>0$ such that for any $(g,v) \in
X(u_0)$ with $|F_g| = |B_{r}|$ and $||g- r  ||_{C^{1,1}(\R)}\leq \delta $ it holds
\[
\F(g,v) \geq \F(r, u_r),
\]
where $u_r$ stands for the elastic equilibrium associated to the disk $B_r$.
\end{remark}

The previous remark enables us to prove the global minimality of the disk when the volume of the annulus  is small.

\begin{proposition}
Suppose that $R_0$ is the radius of the large ball and $u_0$ is the Dirichlet boundary conditions as in \eqref{dirichlet} with fixed $\alpha >0$. There exists $r_{glob} < R_0$ such that for every $r \in (r_{glob}, R_0)$ the round configuration, with a disk $B_r$, is a global minimizer of $\F$ under the volume constraint.
\end{proposition}

\begin{proof}
We argue by contradiction and assume that there exist a sequence of radii $r_n\nearrow R_0$ and a
sequence $(k_n,w_n)\in X(u_0)$ of minimizers of $\F$ under the volume constraint 
$\abs{\Omega_{k_n}}=\abs{A(R_0,r_n)}$ such that
\[
 \F(k_n,w_n)<\F(r_n,u_n)\,,
\]
where $u_n$ stands for the elastic equilibrium relative to $r_n$. Since  $(k_n,w_n)$ minimizes $\F$ we immediately have that $\Ha^1(\Gamma_{k_n} \cup \Sigma_{k_n})\to 2\pi R_0$. Therefore, since $F_{k_n}$ is connected, we deduce that $\eps_n := d_\Ha(\Gamma_{k_n} \cup \Sigma_{k_n},\Gamma_{r_n})\to 0$ as $n\to\infty$.

We may calculate the elastic equilibrium 
\[
 u_n(\rho,\theta)= \left(\frac{a_n}{\rho} +b_n \rho\right)   \sigma(\theta),
\]
where
\[
b_n =\left( 1+  \frac{\mu + \lambda}{\mu} \frac{r_n^2}{R_0^2} \right)^{-1}\alpha \quad \text{and} \quad  a_n = \frac{\mu + \lambda}{\mu}r_n^2 \,  b_n.
\]
By Remark \ref{global} we have that for large $n$ it holds
\[
\partial^2 \F(r_n, u_n)[\psi] \geq c_0 || \psi ||_{H^1(\partial B_{r_n})}^2,
\]
for $ \int_0^{2 \pi}\psi \, d \theta = 0$, where $c_0$ is independent of $n$.

We note that $u_n$ is also the elastic equilibrium in the annulus $A(R, r_n)$, for any $R>R_0$, with respect to its own boundary conditions on $\partial B_{R}$, $v(R, \theta) = u_n(R, \theta) $. For $R > R_0$ we define
\[
\F_{R}(g,v) = \int_{B_{R} \setminus F_g} Q(E(v))\, dz + \Ha^1(\Gamma_g) + 2\Ha^1(\Sigma_g)
\]
and 
\[
X_{R}(u_n) = \{   (g,v) \mid g  \in BV_{\sharp}(\R), \, v \in H_{\loc}^1(\R^2 \setminus F_g; \R^2), \, v = u_n \,\,
\text{outside} \,\, B_{R} \}\,.
\]
Consider the estimate \eqref{2nd_1st.est} for $\partial^2 \F_{R_1}(r_n, u_n)[\psi]$, i.e., replace $R_0, r $ and $b$ by $R_1, r_n $ and  $b_n$. By continuity we may choose $R_1$ close to $R_0$ such that  
\[
\partial^2 \F_{R_1}(r_n, u_n)[\psi] \geq \frac{c_0}{2} || \psi ||_{H^1(\partial B_{r_n})}^2,
\]
for $ \int_0^{2 \pi}\psi \, d \theta = 0$. Arguing as in Remark \ref{global} we conclude that $(r_n, u_n)$ is a
local $C^{1,1}$-minimizer of $\F_{R_1}$ uniformly in $n$, i.e., there is $\delta>0$, independent of $n$, such that for
any $(g,v) \in X_{R_1}(u_n)$, with $||g- r_n||_{ C^{1,1}(\R)} < \delta$,  it holds
\begin{equation}
\label{glob.min}
\F_{R_1}(g,v) \geq \F_{R_1}(r_n, u_n). 
\end{equation}

Define
\[
 \tilde{w}_n(z):=\begin{cases}
w_n(z) & \text{if } z \in \bar{B}_{R_0} \setminus F_{k_n}\\
u_n(z) & \text{if } z\in A(R_1,R_0)\,.
\end{cases}
\]
By the assumption on $(k_n, w_n)$ it holds 
\begin{equation}
\label{glob.cont}
\F_{R_1}(k_n,  \tilde{w}_n) <\F_{R_1}(r_n, u_n).
\end{equation} 
Suppose that $(g_n,v_n)$ is a  solution of the problem
\[
\min \{  \F_{R_1}(g,  v) + \Lambda \big|  | F_g| - |B_{r_n}|  \big|   : (g,v) \in X_{R_1}(u_n), \, g \leq r_n + \eps_n   \},
\]
where $\Lambda$ is large. Arguing as in Lemma \ref{lem-balltohaus}, Lemma \ref{lem-more-regular} and  Lemma
\ref{lem-convergence} we conclude that $g_n \to R_0$ in $C^{1,1}(\R)$. 
In particular, $||g_n - r_n ||_{C^{1,1}(\R)} \to 0$. 

By the minimality of $(g_n, v_n)$ we have  that $\F_{R_1}(g_n,  v_n)  + \Lambda \big|  | F_g| - |B_{r_n}|  \big| \leq  \F_{R_1}(k_n,  \tilde{w}_n) $. Defining $\tilde{g}_n = g_n + \delta_n$ such that $|F_{\tilde{g}_n}| = |B_{r_n}| $ we obtain, as in  \eqref{absorbed}, that 
\begin{equation}
\label{scale.vol}
\F_{R_1}(\tilde{g}_n,  v_n) \leq \F_{R_1}(g_n,  v_n)  + \Lambda \big|  | F_g| - |B_{r_n}|  \big|  \leq  \F_{R_1}(k_n,  \tilde{w}_n),
\end{equation}
 when $\Lambda$ is large enough. Moreover  $\delta_n \to 0$. Hence  $||\tilde{g}_n - r_n ||_{C^{1,1}(\R)} \to 0$ and
therefore  \eqref{glob.min}, \eqref{glob.cont} and \eqref{scale.vol} imply
\[
 \F_{R_1}(r_n, u_n) \leq \F_{R_1}(\tilde{g}_n,v_n) \leq    \F_{R_1}(k_n,  \tilde{w}_n) <  \F_{R_1}(r_n, u_n),
\]
which is a contradiction.

\end{proof}

\section*{Acknowledgements}
This research was partially funded by the
2008 ERC Advanced Grant no.\ 226234 \emph{Analytic Techniques for
Geometric and Functional Inequalities} and by the Marie Curie project IRSES-2009-247486 of the Seventh Framework
Programme.

The authors are thankful to N.\ Fusco for introducing the problem and for helpful discussions.

\bibliography{biblio}

\end{document}